\documentclass[11pt]{article}
\usepackage{amssymb,amsmath,setspace,graphicx,multicol,wrapfig,amsfonts,amsthm,titling,url,array,parskip,enumitem,bbm,color,tikz,tikz-cd,caption,mathdots}
\usepackage{mathtools}
\usepackage{placeins,comment}
\usepackage[colorlinks=true, pdfstartview=FitV, linkcolor=blue, citecolor=blue, urlcolor=blue]{hyperref}
\usepackage[capitalize,nameinlink,noabbrev,nosort]{cleveref}
\usepackage{aurical,pbsi}
\usepackage[centertableaux]{ytableau}
\usepackage[T1]{fontenc}
\usepackage[hmargin=2.5cm,vmargin=2.5cm]{geometry}

\numberwithin{equation}{section}

\theoremstyle{plain}
\newtheorem{theorem}{Theorem}[section]
\newtheorem{prop}[theorem]{Proposition}
\newtheorem{lemma}[theorem]{Lemma}
\newtheorem{cor}[theorem]{Corollary}

\theoremstyle{definition}
\newtheorem{defn}[theorem]{Definition}
\newtheorem{example}[theorem]{Example}
\newtheorem{question}[theorem]{Question}
\newtheorem{remark}[theorem]{Remark}
\crefname{defn}{Definition}{Definitions}
\crefname{theorem}{Theorem}{Theorems}
\crefname{lemma}{Lemma}{Lemmas}
\crefname{example}{Example}{Examples}

\newcommand{\calB}{\mathcal{B}}
\newcommand{\wcalB}{\widetilde{\mathcal{B}}}

\newcommand{\abs}[1]{\lvert #1 \rvert}
\newcommand{\BN}{\mathbb{N}}
\newcommand{\ZZ}{\mathbb{Z}}

\newcommand{\RR}{\mathbb{R}}

\newcommand{\fermwords}{\mathcal{W}}
\newcommand{\bosonwords}{\widetilde{\mathcal{W}}}
\newcommand{\fsl}{\mathfrak{sl}}
\newcommand{\asl}{\widehat{\fsl}}
\newcommand{\combR}{\mathcal{R}}
\newcommand{\bwt}{\widetilde{\mathbf{w}}}
\newcommand{\wt}{\widetilde{w}}
\newcommand{\bw}{\mathbf{w}}

\newcommand{\fa}{\mathfrak{a}}
\newcommand{\fb}{\mathfrak{b}}

\newcommand{\D}{D}
\newcommand{\NN}{\mathbbm{N}}
\newcommand{\R}{R}
\newcommand{\Rf}{\widetilde{R}}
\newcommand{\Ff}{\widetilde{F}}

\newcommand{\uu}{\mathbf{u}}
\newcommand{\vv}{\mathbf{v}}
\newcommand{\xx}{\mathbf{x}}
\newcommand{\yy}{\mathbf{y}}
\newcommand{\Dbf}{\mathbf{D}}

\newcommand{\Qbf}{\mathbf{Q}}

\DeclareMathOperator{\Par}{cpair_f}
\DeclareMathOperator{\bPar}{cpair_b}
\DeclareMathOperator{\MLQ}{MLQ}
\DeclareMathOperator{\bMLQ}{bMLQ}
\DeclareMathOperator{\pr}{pr}
\DeclareMathOperator{\rate}{rate}

\DeclareMathOperator{\MLD}{bMLQ}
\DeclareMathOperator{\FM}{FM}

\DeclareMathOperator{\cw}{cw}
\DeclareMathOperator{\bb}{\mathbf{b}}
\DeclareMathOperator{\s}{\mathbf{s}}

\DeclareMathOperator{\Wt}{wt}
\DeclareMathOperator{\w}{w}

\renewcommand{\mid}{:}

\newcommand\tcr[1]{\textcolor{red}{#1}}
\newcommand\tcb[1]{\textcolor{blue}{#1}}

\newcommand{\dfn}[1]{\textbf{#1}}

\newlength\cellsize 
\savebox2{%
\begin{picture}(12,12)
\put(0,0){\line(1,0){12}}
\put(0,0){\line(0,1){12}}
\put(12,0){\line(0,1){12}}
\put(0,12){\line(1,0){12}}
\end{picture}}

\newcommand\cellify[1]{\def\thearg{#1}\def\nothing{}%
\ifx\thearg\nothing
\vrule width0pt height\cellsize depth0pt\else
\hbox to 0pt{\usebox2\hss}\fi%
\vbox to 12\unitlength{
\vss
\hbox to 12\unitlength{\hss$#1$\hss}
\vss}}

\newcommand\tableau[1]{\vtop{\let\\=\cr
\setlength\baselineskip{-16000pt}
\setlength\lineskiplimit{16000pt}
\setlength\lineskip{0pt}
\halign{&\cellify{##}\cr#1\crcr}}}

\savebox3{%
\begin{picture}(12,12)
\put(0,0){\line(1,0){12}}
\put(0,0){\line(0,1){12}}
\put(12,0){\line(0,1){12}}
\put(0,12){\line(1,0){12}}
\end{picture}}

\newcommand\expath[1]{%
\hbox to 0pt{\usebox3\hss}%
\vbox to 12\unitlength{
\vss
\hbox to 12\unitlength{\hss$#1$\hss}
\vss}}

\usepackage[colorinlistoftodos]{todonotes}

\title{Twisted multiline queues for the steady states of TASEP and TAZRP}

\author{Mandelshtam, Olya \\ \texttt{omandels@uwaterloo.ca} \and Scrimshaw, Travis \\ \texttt{tcscrims@gmail.com}}

\begin{document}

\maketitle

\begin{abstract}
We define an algorithm on fermionic and bosonic twisted multiline queues that projects to the multispecies totally asymmetric simple exclusion process (TASEP) and the totally asymmetric zero range process (TAZRP) on a ring, respectively. Our algorithm on fermionic multiline queues generalizes the Ferrari--Martin algorithm for the TASEP, and we show it is equivalent to the algorithm of Arita--Ayyer--Mallick--Prolhac (2011). Our algorithm on bosonic multiline queues is novel and generalizes the corresponding algorithm of Kuniba--Maruyama--Okado (2016) for the TAZRP. We also define a Markov process on bosonic twisted multiline queues that projects to the TAZRP and intertwines with the symmetric group action on the rows of the multiline queues.
\end{abstract}

\section{Introduction}

In this paper, we consider two continuous time Markov chains of particles hopping on a ring represented by a discrete circular lattice $\ZZ_n := \ZZ / n\ZZ$. The first is the \dfn{multispecies totally asymmetric simple exclusion process (TASEP)} on a ring. For a positive integer $n$ and a partition $\lambda=(\lambda_1\geq\cdots\geq\lambda_n\geq 0)$, the states of a TASEP of type $(\lambda,n)$ are words that are permutations of $\{\lambda_1,\ldots,\lambda_n\}$ (equivalently, equal to $\lambda$ as multisets). The transitions are given by swapping adjacent letters $ubav \rightarrow uabv$ if $a>b$, and every such swap is triggered by an exponential clock with rate $1$, for each pair of adjacent sites. In this paper, we will take the representatives $\ZZ_n = \{1,2,\dotsc,n\}$.

It was first found by Angel~\cite{Angel06} in the two species case and later proved by Ferrari--Martin~\cite{FM07} in general that the stationary distribution of the TASEP (on a ring) can be sampled by \dfn{fermionic multiline queues} of partition shape through a projection map that is now known as the Ferrari--Martin (FM) algorithm. It was shown in the work~\cite{KMO15} of Kuniba--Maruyama--Okado that the FM algorithm has another interpretation originating from statistical mechanics. Their key observation was that the algorithm of Nakayashiki--Yamada~\cite[Rule~3.10]{NY95} to describe the combinatorial R matrix between single column Kirillov--Reshetikhin (KR) crystals (corresponding to the affine quantum group of $\fsl_n$; see, \textit{e.g.},~\cite{KKMMNN91}) matches the pairing procedure of the FM algorithm, which allowed the projection map to be recast through a corner transfer matrix. Meanwhile, Arita--Ayyer--Mallick--Prolhac introduced in~\cite{AAMP} a method to understand the fermionic multiline queues and the projection map as an operation on words. Their so-called commutativity conjecture was that a more general class of multiline queues of composition shape obtained by permuting the rows by an action of the symmetric group, which we call \dfn{twisted multiline queues}, could also be used to compute the stationary probabilities of the TASEP. This conjecture was resolved in the affirmative by Aas--Grinberg--Scrimshaw~\cite{AGS20} by showing the natural action of the combinatorial R matrix does not affect the projection.

The other Markov chain we consider in this paper is the \dfn{multispecies totally asymmetric zero range process (TAZRP)} on a ring. For $n,k\geq 1$ and a partition $\lambda = (\lambda_1 \geq \cdots \geq \lambda_k > 0)$, a state of the TAZRP of type $(\lambda, n)$ is a multiset partition of $\lambda$ with $n$ parts (equivalently, a multiset word, where the union of its components, viewed as multisets, equals $\lambda$).  Various possible dynamics for the TAZRP have been studied; in each, at each site $j$ of the lattice, some set of particles may hop to the adjacent site with some rate that depends only the content of site $j$. Kuniba--Maruyama--Okado~\cite{KMO16,KMO16II} studied a version of the TAZRP (we will refer to this version as kTAZRP) using \dfn{bosonic multiline queues} of partition shape as a manifestation of rank-level duality with the fermionic multiline queue. In particular, they defined a bosonic projection map and the matrix product formula to compute the stationary distribution of the kTAZRP. Recently, Ayyer--Mandelshtam--Martin studied in~\cite{AMM22} a different TAZRP model with additional parameters $t,x_1,\ldots,x_n$, which we refer to as the $t$-TAZRP, and the $0$-TAZRP when $t=0$. They showed in~\cite{AMM20} that the normalization constant for the steady state distribution (also known as the partition function) of the $t$-TAZRP is given by the modified Macdonald polynomials (specialized at $q = 1$). At $x_1=\cdots=x_n=1$, the stationary distribution of the $0$-TAZRP coincides with that of the kTAZRP, and can be computed with weighted bosonic multiline queues through the same projection map as~\cite{KMO16}.

In this article, we describe a new projection map on both the fermionic and bosonic twisted multiline queues (of composition shape), by treating them as operators on words, that yields all possible states in TASEP and TAZRP, respectively. Our projection map on fermionic twisted multiline queues is equivalent to that of~\cite{AAMP} (\cref{thm:equivalent_proj_AAMP}), and thus it is equivalent to the FM algorithm and that of~\cite{KMO15} for fermionic multiline queues of partition shape. However, our approach naturally generalizes to the bosonic case, while that of~\cite{AAMP} does not. For a bosonic multiline queue of partition shape, our projection map coincides with the one in~\cite{KMO16} (\cref{lem:straight_boson_FM}). Our main result is that our projection maps are invariant of the natural action of the combinatorial R matrix (\cref{thm:pi invariant of R,cor:bosonic Phi sigma invariant}), which means they can be used to compute the stationary distribution of the TASEP and $0$-TAZRP, respectively, using twisted multiline queues. This invariance is shown using a generalization of the corner transfer matrix description in~\cite{KMO16}.
Consequently, we have another proof of~\cite[Thm.~3.1]{AGS20} that only uses the properties of the corner transfer matrix and combinatorial R matrix in contrast to the combinatorial analysis of~\cite{AGS20} (\cref{thm:projection is KMO,thm:bosonic projection is KMO}).

There is a well-known Markov chain on fermionic multiline queues defined using \dfn{ringing paths} that projects to the TASEP~\cite{FM07} via the FM projection map. This is the essential component in the Ferrari--Martin construction of the steady state distribution of the TASEP. We introduce a bosonic analogue of ringing paths to define a new Markov chain on bosonic twisted multiline queues that projects to the $0$-TAZRP. Our Markov process is different from the one obtained by specializing results of~\cite{AMM22}. We show that our Markov chain on bosonic multiline queues interlaces with the combinatorial R matrix (equivalently, the symmetric group action).
On the other hand, we exhibit an explicit example where na\"ively lifting ringing paths to fermionic twisted multiline queues does not project to TASEP (\cref{rem:femion_MLQ_ringing}).

The article is organized as follows.
In \cref{sec:definitions} we give the necessary background on classical multiline queues, the combinatorial R matrix, and the TASEP and TAZRP Markov chains. 
In \cref{sec:projection map}, we define the fermionic and bosonic projection maps.
In \cref{sec:R}, we formulate these maps in terms of a corner transfer matrix built from combinatorial R matrices in order to prove our results.
In \cref{sec:GY}, we show that our procedure on fermionic multiline queues is equivalent to the projection map defined in~\cite{AAMP}.
In \cref{sec:bMLQ Markov}, we define a Markov chain on bosonic twisted multiline queues that projects to the 0-TAZRP.
Finally, we give an alternate presentation of our projection maps in \cref{appendix}.

\subsection*{Acknowledgements}

OM was supported by NSERC grant RGPIN-2021-02568.
TS was partially supported by Grant-in-Aid for Scientific Research for Early-Career Scientists 23K12983.


\section{Definitions}\label{sec:definitions}

Let $\BN$ denote the set of positive integers.
For this section, we fix some $k, n \in \BN$.
Define $\mathcal{M}(n)$ to be the set of multisets with entries in $[n] = \{1, 2, \ldots, n\}$.
A partition $\lambda = (\lambda_1 \geq \lambda_2 \geq \cdots \geq \lambda_k \geq 0)$ is a weakly decreasing finite sequence of nonnegative integers, and the length of $\lambda$, denoted $\ell(\lambda)$ is the largest $\ell$ such that $\lambda_{\ell} > 0$.
A (weak) composition is a finite sequence of nonnegative integers.

\subsection{Multiline queues}

Let $\alpha=(\alpha_1,\ldots,\alpha_k)$ be a (weak) composition. A \dfn{fermionic multiline queue} of type $(\alpha, n)$ is defined to be a tuple $\Qbf = (Q_1,\ldots,Q_k)$, where $Q_j\subseteq [n]$ and $\abs{Q_j} = \alpha_j$. There is a natural correspondence with binary $n \times k$ matrices (\textit{i.e.}, entries are either $0$ or $1$) with row sums $\alpha$, where $Q_j$ is the set of nonzero entries in row $j$.
A \dfn{bosonic multiline queue} of type $(\alpha,n)$ is a tuple $\Dbf = (D_1,\ldots,D_k)$, where $D_j\in \mathcal{M}(n)$ and $\abs{D_j}=\alpha_j$. Likewise, there is a natural correspondence with nonnegative integer $n \times k$ matrices, where the entry $a_{ij} \geq 0$ counts the number of $j$s in $D_i$.
When $\alpha$ is a partition, we call $\Qbf$ (resp.\ $\Dbf$) a \dfn{straight} fermionic (resp.\ bosonic) multiline queue, and we call it \dfn{twisted} otherwise. Denote the sets of fermionic and bosonic multiline queues of type $(\alpha,n)$ by $\MLQ(\alpha,n)$ and $\bMLQ(\alpha,n)$, respectively.

For a sequence of indeterminates $\xx = (x_1,x_2,\ldots,x_n)$ (in what follows, we will generally specialize these to elements in $\RR_{>0}$), define the \dfn{weight} of a (fermionic or bosonic) multiline queue $\Qbf$ to be the monomial $\xx^{\Qbf}\coloneqq \prod_{i=1}^n x_i^{m_i(\Qbf)}$, where $m_i(\Qbf)$ is the multiplicity of $i$ in~$\Qbf$.

Pictorially, we represent a (fermionic or bosonic) multiline queue $\Dbf=(D_1,\ldots,D_k)$ using a dot diagram, by placing $\ell$ dots in column $j$ of row $i$ of an $n\times k$ array, where $\ell$ is the number of $j$'s in $D_i$. If $\Dbf$ is fermionic, $\ell\in\{0,1\}$, so the dot diagram is also fermionic. Otherwise, the dot diagram is bosonic. For instance, if $\alpha=(2,3,1)$ and $n=4$, we show a fermionic multiline queue $\Qbf$ (left) and a bosonic one $\Dbf$ (right), both of type $(\alpha,n)$ with respective weights $\xx^{\Qbf} = x_1x_2^2x_3x_4^2$ and $\xx^{\Dbf}=x_1^2x_2^2x_3^3$:
\[\ytableausetup{boxsize=1em}
\Qbf=(\{2\},\{2,3,4\},\{1,4\})=\ytableaushort{\bullet{}{}\bullet,{}\bullet\bullet\bullet,{}\bullet{}{}}
\qquad\text{ and }\qquad
\Dbf=(\{2\},\{2,3,3\},\{1,1\}) = \ytableaushort{{\bullet\bullet}{}{}{},{}{\bullet}{\bullet\bullet}{},{}\bullet{}{}}\,.
\]

Call $\fermwords = \{(a_1,\ldots,a_n) \mid a_j\in\NN\}$ the set of \dfn{fermionic words} and $\bosonwords=\{(A_1,\ldots,A_n) \mid A_j\in \mathcal{M}(k)\}$ the set of \dfn{bosonic words}. To emphasize that we are working with words, we will often use multiplicative notation, writing $w_1 \cdots w_n$ for $(w_1, \ldots, w_n)\in\fermwords$ and $a_1\cdots a_k$ for $A_j=\{a_1,\ldots,a_k\}$ for each $j$ in $(A_1,\ldots,A_n)\in\bosonwords$ when there is no danger of confusion. For $u\in\fermwords$, define $u_+$ to be the multiset of nonzero elements in $u$. For $j\geq 1$, define $u^{(+j)}\in\fermwords$ (resp.~$U^{(+j)}\in\bosonwords$) to be the word $u$ (resp.~$U$) with each nonzero entry increased by $j$. For instance, let $v=1\,0\,2\,0\,1\,3 \in\fermwords$, and $V=(\emptyset,24,113,\emptyset)\in\bosonwords$. Then we have $v_+=\{1,1,2,3\}$, $v^{(+1)}=2\,0\,3\,0\,2\,4$, $v^{(+2)}=3\,0\,4\,0\,3\,5$, $V^{(+1)}=(\emptyset,35,224,\emptyset)$, $V^{(+2)}=(\emptyset,46,335,\emptyset)$, and so on.

For a fixed $n$, define $\iota \colon \{B \mid B\subseteq [n]\}\rightarrow \{0,1\}^n$ to be the map which sends a subset to its corresponding (fermionic) indicator vector by $\iota(B) = (a_1,\ldots,a_n)$, where $a_j$ is $1$ if $j\in B$ and $0$ otherwise. For two indicator vectors $a,b\in \{0,1\}^n$, we define $a\leq b$ if $a_i\leq b_i$ for all $i$, and we say $a$ and $b$ are \dfn{nested}. Note that $\mathbf{A}\subseteq \mathbf{B}$ if and only if $\iota(\mathbf{A})\leq \iota(\mathbf{B})$. For $\mathbf{a}\in\{0,1\}^n$ and $\bw = (w_1,\ldots,w_n)\in\fermwords$, define $\bw + \mathbf{a} = (w_1+a_1,\ldots,w_n+a_n)$. 

We extend the above definition to multisets for $\iota \colon \mathcal{M}(n)\rightarrow \mathcal{M}(1)^n$. Let $D\in\mathcal{M}(n)$, and define $\iota(D)=(A_1,\ldots,A_n)$, where $A_j \in \mathcal{M}(1)$ is the multiset with a $1$ for each $j \in D$. For two bosonic indicator vectors $A$ and $B$, define $A\leq B$ if $\abs{A_i} \leq \abs{B_i}$ for all $i$; we call such $A$ and $B$ \dfn{nested}. Note that $D\subseteq E \in \mathcal{M}(n)$ if and only if $\iota(D)\leq \iota(E)$. For $C=(C_1,\ldots,C_n)\in\bosonwords$ and $A=(A_1,\ldots,A_n)\in\mathcal{M}(1)^n$ such that $\abs{A_j}\leq \abs{C_j}$, define $C\oplus A=(C_1',\ldots,C'_n)$, where $C_j'$ is equal to $C_j$ with its largest $\abs{A_j}$ entries increased by $1$. See \cref{ex:iota,ex:nested}.

\begin{example}\label{ex:iota}
    Let $n=6$. For $B=\{1,3,4,6\}\subseteq [n]$ and $C=\{3,6\}\subseteq [n]$, $\iota(B)=(1,0,1,1,0,1)$ and $\iota(C)=(0,0,1,0,0,1)$. Then we have $\iota(C)\leq \iota(B)$ and $\iota(B)+\iota(C)=(1,0,2,1,0,2)$. For $D=\{1,1,3,4,4,4\}\in \mathcal{M}(n)$ and $E=\{1,4,4\}\in \mathcal{M}(n)$, we have $\iota(D)=(\{1,1\},\emptyset,\{1\},\{1,1,1\},\emptyset,\emptyset)$ and $\iota(E)=(\{1\},\emptyset,\emptyset,\{1,1\},\emptyset,\emptyset)$. Then we have $\iota(E)\leq \iota(D)$ and $\iota(D)\oplus\iota(E)=(12,\emptyset,1,122,\emptyset,\emptyset)$.
\end{example}

For an integer $m\geq 1$, define $\upsilon_m \colon \fermwords\rightarrow \{0,1\}^n$ by $\upsilon_m(w_1, \cdots ,w_n)=(y_1,\ldots,y_n)$, where $y_j=1$ if $w_j\geq m$ and $y_j=0$ otherwise. Extend the definition to $\upsilon_m \colon \bosonwords\rightarrow \mathcal{M}(1)^n$ by $\upsilon_m( \wt_1, \cdots ,\wt_n) = (A_1,\ldots,A_n)$, where $A_j=\{1,\ldots,1\}\in\mathcal{M}(1)$ has a $1$ for each $j \in \wt_j$ with $j \geq m$. Observe that if $\bw$ is a fermionic or bosonic word, then $\upsilon_m(\bw)\geq \upsilon_{m'}(\bw)$ for any $m\leq m'$.

For $\bw\in\fermwords$ (with entries in $[k]$), we have the unique decomposition of $\bw$ into (fermionic) nested indicator vectors:
\[
\bw=\upsilon_1(\bw)+\upsilon_2(\bw)+\cdots+\upsilon_k(\bw).
\]
Similarly, for $\bwt \in \bosonwords$ (with elements in $\mathcal{M}(k)$), we have the unique decomposition of $\bwt$ into (bosonic) nested indicator vectors:
\[
\bwt=\upsilon_1(\bwt)\oplus\upsilon_2(\bwt)\oplus\cdots\oplus \upsilon_k(\bwt).
\]
Note that $\oplus$ is not defined unless the right factor is a (bosonic) indicator vector, so there is only one possible well-defined order to apply the $\oplus$ operators.

\begin{example}\label{ex:nested}
    For $\bw = 3252035 \in \fermwords$, we have 
    \begin{align*}
    \upsilon_1(\bw)=\upsilon_2(\bw)&=(1,1,1,1,0,1,1),\\
    \upsilon_3(\bw)&=(1,0,1,0,0,1,1),\\
    \upsilon_4(\bw)=\upsilon_5(\bw)&=(0,0,1,0,0,0,1).
    \end{align*}
    For $\bwt=(233,\emptyset,2235,25)\in\bosonwords$ we have: 
    \begin{align*}
    \upsilon_1(\bwt)=\upsilon_2(\bwt)&=(\{1,1,1\},\emptyset,\{1,1,1,1\},\{1,1\}),\\
    \upsilon_3(\bwt)&=(\{1,1\},\emptyset,\{1,1\},\{1\})\\
    \upsilon_4(\bwt)=\upsilon_5(\bwt)&=(\emptyset,\emptyset,\{1\},\{1\}).
    \end{align*}
\end{example}

\begin{remark}
    It is useful to think of the decomposition into nested indicator vectors in terms of dot diagrams. A word $\bw=w_1\ldots w_n \in\fermwords$ can be uniquely represented by a dot diagram (which is different from that for MLQs) by placing $w_i$ bottom-justified dots in the $i$'th column, as in the example below. Then $\upsilon_j(\bw)$ is read off the $j$'th row of the dot diagram (from bottom to top), with dots corresponding to 1's. 
    \[
    \ytableausetup{boxsize=1em}
    \text{For $\bw = 3\,2\,5\,2\,0\,3\,5$, the dot diagram is } \quad 
 \ytableaushort{{}{}{\bullet}{}{}{}{\bullet},{}{}{\bullet}{}{}{}{\bullet},{\bullet}{}{\bullet}{}{}{\bullet}{\bullet},{\bullet}{\bullet}{\bullet}{\bullet}{}{\bullet}{\bullet},{\bullet}{\bullet}{\bullet}{\bullet}{}{\bullet}{\bullet}}\,.
    \]
    
    For $\bwt=(\wt_1,\ldots,\wt_n)\in\bosonwords$, we can similarly represent $\bwt$ by a (bosonic) dot diagram. We shall record each multiset $\wt_i\in \mathcal{M}$ as a partition $\lambda(\wt_i)$ whose parts are the entries of $\wt_i$ sorted in decreasing order. Then the number of dots in column $i$ at row $j$ is given by the $j$'th part of the conjugate partition $\lambda(\wt_i)'$. Again, $\upsilon_j(\bwt)$ is read off the $j$'th row of the dot diagram.
    \[
    \ytableausetup{boxsize=1em}
    \text{For $\bwt=(233,\emptyset,2235,25)$, the dot diagram is }
    \quad 
    \ytableaushort{{}{}{\bullet}{\bullet},{}{}{\bullet}{\bullet},{\bullet\bullet}{}{\bullet\bullet}{\bullet},{\substack{\bullet\\[-2pt]\bullet\bullet}}{}{\substack{\bullet\bullet\\[-2pt]\bullet\bullet}}{\bullet\bullet},{\substack{\bullet\\[-2pt]\bullet\bullet}}{}{\substack{\bullet\bullet\\[-2pt]\bullet\bullet}}{\bullet\bullet}}\,.
    \]
\end{remark}

Next we define the \dfn{cylindrical pairing rule}, which has both fermionic and bosonic versions. 

\begin{defn}[Cylindrical pairing rule]\label{def:bracketing}
Let $n$ be a positive integer. For two subsets $A,B\subseteq [n]$ representing fermionic configurations of particles on the lattice $\ZZ_n$, define $\Par(A,B)$ to mean cylindrically pairing particles \emph{weakly to the right} from row $B$ above to row $A$ below. Formally, this is done by running through the elements $1\leq j\leq n$, and for each $j$ recording a open bracket ``$($'' if $j\in B$ and a closed bracket ``$)$'' if $j\in A$, in that order, and then applying the cylindrical bracketing rule (recursively pairing each pair of open and closed brackets, considered on the cylinder, with no unpaired brackets between them). 

Similarly, for two multisets $\widetilde{A},\widetilde{B}\subseteq \mathcal{M}(n)$ representing bosonic configurations of particles, we use the notation $\bPar(\widetilde{A},\widetilde{B})$ to mean cylindrically pairing particles \emph{strictly to the left} from row $\widetilde{B}$ above to row $\widetilde{A}$ below. This is done by running through the elements $1\leq j\leq n$, and recording a ``$)$'' for each instance of $j\in \widetilde{B}$ and a ``$($'' for each instance of $j\in \widetilde{A}$, in that order, and then applying the cylindrical bracketing rule. 
\end{defn}
If $a,b\in\{0,1\}^n$, we will slightly abuse notation by writing $\Par(a,b)$ to mean $\Par(\iota^{-1}(a),\iota^{-1}(b))$, and similarly for the bosonic analogue $\bPar$.

It should be noted that the reversed pairing direction for the bosonic pairing rule precisely corresponds to recording the particles in reverse order and swapping the brackets in the definition of $\bPar$ as opposed to $\Par$.

For $y\in [n]$, we say the particle $y$ in $B$ is paired in $\Par(A,B)$ if $y\in B$ and its corresponding open bracket is paired in $\Par(A,B)$, and unpaired if the open bracket is unpaired. Similarly, we say $y$ in $A$ is paired in $\Par(A,B)$ if $y\in A$ and its corresponding closed bracket is paired in $\Par(A,B)$, and unpaired if the closed bracket is unpaired.. The same terminology applies for $\bPar(\widetilde{A},\widetilde{B})$, except in this case there may be multiple paired or unpaired particles at each site of $\widetilde{A}$ or $\widetilde{B}$.

\begin{example}
    \label{ex:cpair}
    For fermionic configurations $A=\{1,2,4\}$, $B=\{1,3,5,6\}$, and $C=\{2,3\}$, the unpaired particles in $\Par(A,B)$ and $\Par(B,C)$ correspond to the red underlined parentheses in $())()\underline{\tcr{(}}($ and $\underline{\tcr{)}}(())\underline{\tcr{)}}$, respectively.  We visualize the pairings by drawing the pairing lines weakly to the right (the unpaired particles are colored red):
    
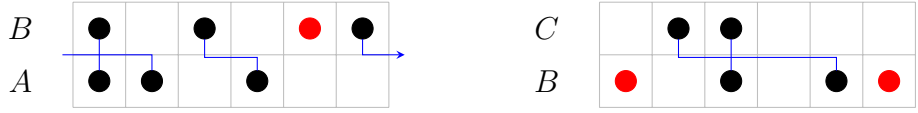
\[
\begin{tikzpicture}[scale=0.7]
\def \w{1};
\def \h{1};
\def \r{0.2};
    
\begin{scope}[xshift=0cm]
\node at (-1,1.5) {\large $B$};
\node at (-1,.5) {\large $A$};
\draw[gray!50,thin,step=\w] (0,0) grid (6*\w,2*\h);
\foreach \xx\yy\c in {0/1/black,2/1/black,4/1/red,5/1/black,0/0/black,1/0/black,3/0/black}
    {
    \filldraw[\c] (\w*.5+\w*\xx,\h*.5+\h*\yy) circle (\r cm);
    }
\draw[blue] (\w*.5,\h*1.5-\r)--(\w*.5,\h*.5+\r);
\draw[blue] (\w*2.5,\h*1.5-\r)--(\w*2.5,\h*.95)--(\w*3.5,\h*.95)--(\w*3.5,\h*.5+\r);    
\draw[blue,-stealth] (\w*5.5,\h*1.5-\r)--(\w*5.5,\h*1.0)--(\w*6.3,\h*1.0);
\draw[blue] (-.2,\h*1.0)--(\w*1.5,\h*1.0)--(\w*1.5,\h*.5+\r);
\end{scope}
\begin{scope}[xshift=10cm]
\node at (-1,1.5) {\large $C$};
\node at (-1,.5) {\large $B$};
\draw[gray!50,thin,step=\w] (0,0) grid (6*\w,2*\h);
\foreach \xx\yy\c in {1/1/black,2/1/black,0/0/red,2/0/black,4/0/black,5/0/red}
    {
    \filldraw[\c] (\w*.5+\w*\xx,\h*.5+\h*\yy) circle (\r cm);
    }
\draw[blue] (\w*2.5,\h*1.5-\r)--(\w*2.5,\h*.5+\r);
\draw[blue] (\w*1.5,\h*1.5-\r)--(\w*1.5,\h*.95)--(\w*4.5,\h*.95)--(\w*4.5,\h*.5+\r); 
\end{scope}
\end{tikzpicture}
\]
    Similarly, for bosonic configurations  $\widetilde{A}=\{1,2,2,4,5\}$, $\widetilde{B}=\{2,2\}$, and $\widetilde{C}=\{1,2,4,6\}$ the unpaired particles in $\bPar(\widetilde{A},\widetilde{B})$ and $\bPar(\widetilde{B},\widetilde{C})$ correspond to the red underlined parentheses in $\underline{\tcr{((}}(())\underline{\tcr{(}}$ and $\underline{\tcr{))}}(())$, respectively. We visualize the pairings by drawing the pairing lines strictly to the left:
 \begin{center}
   \resizebox{!}{1cm}
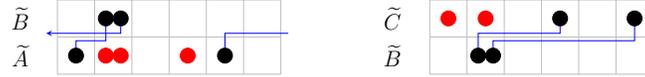
{
\begin{tikzpicture}[scale=0.7]
\def \w{1};
\def \h{1};
\def \r{0.2};
    
\begin{scope}[xshift=0cm]
\node at (-1,1.5) {\large $\widetilde{B}$};
\node at (-1,.5) {\large $\widetilde{A}$};
\draw[gray!50,thin,step=\w] (0,0) grid (6*\w,2*\h);
\foreach \xx\yy\c in {.8/1/black,1.2/1/black,0/0/black,.8/0/red,1.2/0/red,3/0/red,4/0/black}
    {
    \filldraw[\c] (\w*.5+\w*\xx,\h*.5+\h*\yy) circle (\r cm);
    }

\draw[blue] (\w*.5+\w*.8,\h*1.5-\r)--(\w*.5+\w*.8,\h*.9)--(\w*.5,\h*.9)--(\w*.5,\h*.5+\r);    
\draw[blue,-stealth] (\w*.5+\w*1.2,\h*1.5-\r)--(\w*.5+\w*1.2,\h*1.1)--(-.3,\h*1.1);
\draw[blue] (\w*6.2,\h*1.1)--(\w*4.5,\h*1.1)--(\w*4.5,\h*.5+\r);

\end{scope}

\begin{scope}[xshift=10cm]
\node at (-1,1.5) {\large $\widetilde{C}$};
\node at (-1,.5) {\large $\widetilde{B}$};
\draw[gray!50,thin,step=\w] (0,0) grid (6*\w,2*\h);
\foreach \xx\yy\c in {0/1/red,1/1/red,3/1/black,5/1/black,.8/0/black,1.2/0/black}
    {
    \filldraw[\c] (\w*.5+\w*\xx,\h*.5+\h*\yy) circle (\r cm);
    }
    
\draw[blue] (\w*3.5,\h*1.5-\r)--(\w*3.5,\h*1.1)--(\w*.5+\w*.8,\h*1.1)--(\w*.5+\w*.8,\h*.5+\r);
\draw[blue] (\w*5.5,\h*1.5-\r)--(\w*5.5,\h*.9)--(\w*.5+\w*1.2,\h*.9)--(\w*.5+\w*1.2,\h*.5+\r);

\end{scope}
\end{tikzpicture}
}
\end{center}
\end{example}

Now we define the row-swapping involution $\sigma_i$ on fermionic and bosonic multiline queues.

\begin{defn}[The involution $\sigma_i$]
For a fermionic multiline queue $\Qbf = (Q_1,\ldots,Q_k)$, define $\sigma_i$ to be the row-swapping involution acting on $\Qbf$ by sending all cylindrically unmatched entries in $\Par(Q_{i},Q_{i+1})$ from row $i$ to row $i+1$ and vice versa. Similarly, for a bosonic multiline queue $\Dbf=(D_1,\ldots,D_k)$, define $\sigma_i$ to be the row-swapping involution acting on $\Dbf$ by sending all cylindrically unmatched entries in $\bPar(D_{i},D_{i+1})$ 
from row $i$ to row $i+1$ and vice versa. 
\end{defn}

\begin{example}\label{ex:sigma}
Consider the fermionic multiline queue $\Qbf = (\{1,2,4\},\{1,3,5,6\},\{2,3\}) = (Q_1, Q_2, Q_3)$. 
Then we have $\Par(Q_1,Q_2) = \Par(A,B)$ and $\Par(Q_2,Q_3) = \Par(B,C)$ from \cref{ex:cpair}. Thus, we compute $\sigma_1(\Qbf)$ and $\sigma_2(\Qbf)$ by swapping the particles corresponding to the unmatched parentheses between the rows, indicating swapped particles in red:
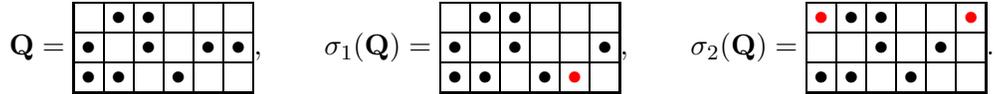
\[
\ytableausetup{boxsize=1em}
\Qbf = \ytableaushort{{}{\bullet}{\bullet}{}{}{},{\bullet}{}{\bullet}{}{\bullet}{\bullet},{\bullet}{\bullet}{}{\bullet}{}{}}, \qquad
\sigma_1(\Qbf) = \ytableaushort{{}{\bullet}{\bullet}{}{}{},{\bullet}{}{\bullet}{}{}{\bullet},{\bullet}{\bullet}{}{\bullet}{\tcr{\bullet}}{}}, \qquad
\sigma_2(\Qbf) = \ytableaushort{{\tcr{\bullet}}{\bullet}{\bullet}{}{}{\tcr{\bullet}},{}{}{\bullet}{}{\bullet}{},{\bullet}{\bullet}{}{\bullet}{}{}}.
\]
Similarly, if we now consider the bosonic multiline queue $\Dbf=(\{1,2,2,4,5\},\{2,2\},\{1,2,4,6\})$, then we have $\bPar(D_1,D_2)=\bPar(\widetilde{A},\widetilde{B})$ and $\bPar(D_2,D_3)=\bPar(\widetilde{B},\widetilde{C})$ from \cref{ex:cpair}. Thus, we compute $\sigma_1(\Dbf)$ and $\sigma_2(\Dbf)$: 
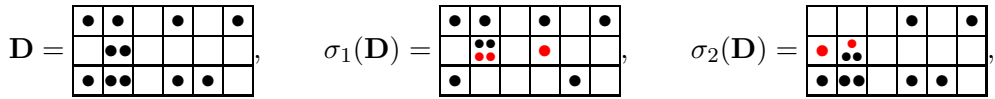
\[
\Dbf=\ytableaushort{{\bullet}{\bullet}{}{\bullet}{}{\bullet},{}{\bullet\bullet}{}{}{}{},{\bullet}{\bullet\bullet}{}{\bullet}{\bullet}{}},
\qquad
\sigma_1(\Dbf)=\ytableaushort{{\bullet}{\bullet}{}{\bullet}{}{\bullet},{}{\substack{\bullet\bullet\\[-2pt]\tcr{\bullet\bullet}}}{}{\tcr{\bullet}}{}{},{\bullet}{}{}{}{\bullet}{}},
\qquad
\sigma_2(\Dbf)=\ytableaushort{{}{}{}{\bullet}{}{\bullet},{\tcr{\bullet}}{\substack{\tcr{\bullet}\\[-2pt]\bullet\bullet}}{}{}{}{},{\bullet}{\bullet\bullet}{}{\bullet}{\bullet}{}},
\]
\end{example}

\subsection{TASEP and TAZRP}

For a partition $\lambda$ and an integer $n$, define $\fermwords(\lambda,n)\subset \fermwords$ to be the set of fermionic words of length $n$ whose nonzero elements rearrange to $\lambda$.
Likewise, define $\bosonwords(\lambda,n)\subset \bosonwords$ to be the set bosonic words of length $n$ such that the (multiset) union of the letters sorts to $\lambda$.
For instance, if $\lambda=(2,1)$, then
\begin{align*}
\fermwords(\lambda,3) & = \{(2,1,0),(1,2,0),(1,0,2),(2,0,1),(0,1,2),(0,2,1)\},
\\
\bosonwords(\lambda,2) & = \{(12,\emptyset),(2,1),(1,2),(\emptyset,12)\}.
\end{align*}
The words represent configurations of particles of different weights, or species. In our convention, the larger labels correspond to particles of higher weight, which will be considered to have higher priority. We may then represent the states of the TASEP (resp.~TAZRP) of type $(\lambda,n)$ by $\fermwords(\lambda,n)$ (resp.~$\bosonwords(\lambda,n)$). The indices of elements in both $\fermwords(\lambda,n)$ and $\bosonwords(\lambda,n)$ are considered on the circle $\ZZ_n$, so that $n+1 \equiv 1$ as indices.

We will define the (multispecies) TASEP of type $(\lambda,n)$ to be the continuous time Markov chain on $\fermwords(\lambda,n)$ with transitions from state $\bw$ to state $\bw'$ having rate $1$ if $\bw=xbay$ and $\bw'=xaby$ with $a>b$. Note that with this convention, heavier particles hop to the left (counterclockwise). Transitions between any other two states have rate $0$. We write the transition rate between two states $\bw,\bw'$ as $\rate(\bw,\bw')$. The \dfn{stationary distribution} can be described as the unique function $\pr \colon \fermwords(\lambda,n) \rightarrow [0,1]$ with $\sum_{\bw\in\fermwords(\lambda,n)}\pr(\bw)=1$ that satisfies the \dfn{balance equation}
\begin{equation}
\label{eq:balance}
\pr(\bw)\sum_{\bw'\in\fermwords(\lambda,n)}\rate(\bw,\bw')=\sum_{\bw'\in\fermwords(\lambda,n)}\pr(\bw)\rate(\bw',\bw)
\end{equation}
for each state $\bw\in\fermwords(\lambda,n)$. 

The 0-TAZRP of type $(\lambda,n)$ is a Markov chain on $\bosonwords(\lambda,n)$ with site-dependent parameters $\{x_1,x_2,\ldots,x_n\}$ where $x_i\in\RR_{>0}$ for $1\leq i\leq n$. In the 0-TAZRP, transitions from state $\bwt=\wt_1\cdots \wt_n$ to state $\bwt'=\wt'_1\cdots \wt'_n$ have nonzero rate if and only if there is some $1\leq j\leq n$ and $y=\max\wt_j$ such that
\[
\wt'_i=\begin{cases}
    \wt_j\setminus\{y\}& \text{if } i=j,\\
    \wt_{j+1}\cup\{y\}& \text{if } i=j+1\\
    \wt_i& \text{if } i\neq j,j+1.
\end{cases}
\] 
 For such $\bwt'$, we have $\rate(\bwt,\bwt')=  x_j^{-1}$. In other words, for $1\leq j\leq n$, there is a possible transition at site $j$, corresponding to the highest labelled particle from site $j$ hopping to site $j+1$, with rate $x_j^{-1}$ given by the parameter associated to site $j$. 

\begin{example}
    The possible transitions from the state $\bwt=(233, \emptyset, 2235, 25)$ are to the state $(23, 3, 223, 5, 25)$ with rate $x_1^{-1}$, to the state $(233, \emptyset, 223, 255)$ with rate $x_3^{-1}$, and to the state $(2335, \emptyset, 2235, 2)$ with rate $x_4^{-1}$.
\end{example}

The stationary distribution $\widetilde{\pr} \colon \bosonwords(\lambda,n)\rightarrow [0,1]$ with $\sum_{\bwt\in\bosonwords(\lambda,n)}\widetilde{\pr}(\bwt)=1$ of the 0-TAZRP is the unique function that satisfies an analogous balance equation to~\eqref{eq:balance}.

\subsection{Ferrari--Martin algorithm}\label{sec:FM}

Fix a partition $\lambda$, and let $\lambda'$ denote the conjugate partition. Ferrari and Martin introduced the \dfn{Ferrari--Martin (FM) algorithm}~\cite{FM07} to define a projection map $\Phi_{\FM} \colon \MLQ(\lambda,n)\rightarrow\fermwords(\lambda',n)$ to compute the stationary probabilities of the TASEP of type $(\lambda',n)$ using straight multiline queues. 

\begin{defn}\label{def:FM}
Let $\Qbf=(Q_1,\ldots,Q_L)\in\MLQ(\lambda,n)$. For row $r=L,L-1,\ldots,1$, label all unlabelled particles with ``$r$''. Then, for $\ell=L,L-1,\ldots,r$, cylindrically pair all particles with label $\ell$ weakly to the right to unpaired particles in row $r-1$, and pass the label ``$\ell$'' to the newly paired particles. When the procedure is completed for row $r=1$, let $\Phi_{\FM}(\Qbf)$ be the word obtained by scanning the sites in the bottom row from left to right and recording a $0$ if the site is empty, and the particle label otherwise.
\end{defn}

It was found in~\cite{KMO16} that by modifying the pairing rule, a similar pairing procedure could be defined for bosonic multiline queues to define the projection map $\widetilde{\Phi}_{\FM}\colon\bMLQ(\lambda,n)\rightarrow \bosonwords$ to compute the stationary probabilities of the kTAZRP. The algorithm is identical except that particles are paired strictly to the left instead of weakly to the right.

Note that the procedures described above are well-defined only when at each row $r>1$ of the multiline queue, there are always at least as many particles in row $r-1$ as in $r$; that is, only when the multiline queue is straight.

\begin{example}\label{ex:FM algorithm}
    We show the sequence of steps of the FM pairing procedure for a fermionic and bosonic multiline queue on three rows for $\lambda=(4,3,2)$. Let $\Qbf=(\{1,3,4,5\},\{2,3,4\},\{3,5\})\in\MLQ(\lambda,5)$ and $\Dbf=(\{1,3,3,5\},\{2,2,4\},\{1,2\})\in\bMLQ(\lambda,5)$.

For $r=3$, we start by labelling all particles in row 3 with ``3''. Next, for $\ell=3$, pair particles with label $\ell$ in row 3 to row 2 weakly to the right in $\Qbf$ and strictly to the left in $\Dbf$:

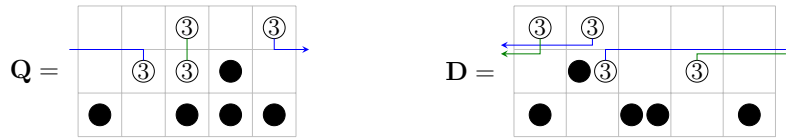
\begin{center}
\resizebox{!}{1.75cm}{
\begin{tikzpicture}[scale=0.7]
\def \w{1};
\def \h{1};
\def \r{0.25};
    
\begin{scope}[xshift=0cm]
\node at (-1,1.5) {\large $\Qbf=$};
\draw[gray!50,thin,step=\w] (0,0) grid (5*\w,3*\h);
\foreach \xx\yy\i\c in {2/2/3/white,4/2/3/white,1/1/3/white,2/1/3/white,3/1/2/black,0/0/1/black,2/0/3/black,3/0/3/black,4/0/2/black}
    {
    \draw[fill=\c] (\w*.5+\w*\xx,\h*.5+\h*\yy) circle (\r cm);
    \node at (\w*.5+\w*\xx,\h*.5+\h*\yy) {\i};
    }

\draw[black!50!green] (\w*2.5,\h*2.5-\r)--(\w*2.5,\h*1.5+\r);
\draw[blue,-stealth] (\w*4.5,\h*2.5-\r)--(\w*4.5,\h*2.0)--(\w*5.3,\h*2.0);
\draw[blue] (-.2,\h*2.0)--(\w*1.5,\h*2.0)--(\w*1.5,\h*1.5+\r);
\end{scope}

\begin{scope}[xshift=10cm]
\def \w{1.2};
\node at (-1,1.5) {\large $\Dbf=$};
\draw[gray!50,thin,xstep=\w,ystep=\h] (0,0) grid (5*\w,3*\h);
\foreach \xx\yy\i\c in {0/2/3/white,1/2/3/white,.75/1/2/black,1.25/1/3/white,3/1/3/white,0/0/3/black,1.75/0/1/black,2.25/0/3/black,4/0/2/black}
    {
    \draw[fill=\c] (\w*.5+\w*\xx,\h*.5+\h*\yy) circle (\r cm);
    \node at (\w*.5+\w*\xx,\h*.5+\h*\yy) {\i};
    }

\draw[blue,-stealth] (\w*1.5,\h*2.5-\r)--(\w*1.5,\h*2.1)--(-.3,\h*2.1);
\draw[blue] (\w*5.2,\h*2)--(\w*1.75,\h*2)--(\w*1.75,\h*1.5+\r);
\draw[black!50!green,-stealth] (\w*.5,\h*2.5-\r)--(\w*.5,\h*1.9)--(-.3,\h*1.9);
\draw[black!50!green] (\w*5.2,\h*1.9)--(\w*3.5,\h*1.9)--(\w*3.5,\h*1.5+\r);

\end{scope}
\end{tikzpicture}
}
\end{center}
    
    Next, for $r=2$, label all unlabelled particles in row 2 with ``2''. Pair particles with label $\ell=3$ in row 2 to unpaired particles in row 1:

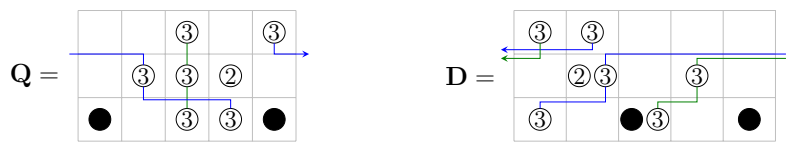
\begin{center}
\resizebox{!}{1.75cm}{
\begin{tikzpicture}[scale=0.7]
\def \w{1};
\def \h{1};
\def \r{0.25};
    
\begin{scope}[xshift=0cm]
\node at (-1,1.5) {\large $\Qbf=$};
\draw[gray!50,thin,step=\w] (0,0) grid (5*\w,3*\h);
\foreach \xx\yy\i\c in {2/2/3/white,4/2/3/white,1/1/3/white,2/1/3/white,3/1/2/white,0/0/1/black,2/0/3/white,3/0/3/white,4/0/2/black}
    {
    \draw[fill=\c] (\w*.5+\w*\xx,\h*.5+\h*\yy) circle (\r cm);
    \node at (\w*.5+\w*\xx,\h*.5+\h*\yy) {\i};
    }

\draw[black!50!green] (\w*2.5,\h*2.5-\r)--(\w*2.5,\h*1.5+\r);
\draw[blue,-stealth] (\w*4.5,\h*2.5-\r)--(\w*4.5,\h*2.0)--(\w*5.3,\h*2.0);
\draw[blue] (-.2,\h*2.0)--(\w*1.5,\h*2.0)--(\w*1.5,\h*1.5+\r);
    
\draw[black!50!green] (\w*2.5,\h*1.5-\r)--(\w*2.5,\h*.5+\r);
\draw[blue] (\w*1.5,\h*1.5-\r)--(\w*1.5,\h*.95)--(\w*3.5,\h*.95)--(\w*3.5,\h*.5+\r);

\end{scope}

\begin{scope}[xshift=10cm]
\def \w{1.2};
\node at (-1,1.5) {\large $\Dbf=$};
\draw[gray!50,thin,xstep=\w,ystep=\h] (0,0) grid (5*\w,3*\h);
\foreach \xx\yy\i\c in {0/2/3/white,1/2/3/white,.75/1/2/white,1.25/1/3/white,3/1/3/white,0/0/3/white,1.75/0/1/black,2.25/0/3/white,4/0/2/black}
    {
    \draw[fill=\c] (\w*.5+\w*\xx,\h*.5+\h*\yy) circle (\r cm);
    \node at (\w*.5+\w*\xx,\h*.5+\h*\yy) {\i};
    }

\draw[blue,-stealth] (\w*1.5,\h*2.5-\r)--(\w*1.5,\h*2.1)--(-.3,\h*2.1);
\draw[blue] (\w*5.2,\h*2)--(\w*1.75,\h*2)--(\w*1.75,\h*1.5+\r);
\draw[black!50!green,-stealth] (\w*.5,\h*2.5-\r)--(\w*.5,\h*1.9)--(-.3,\h*1.9);
\draw[black!50!green] (\w*5.2,\h*1.9)--(\w*3.5,\h*1.9)--(\w*3.5,\h*1.5+\r);
    
\draw[black!50!green] (\w*3.5,\h*1.5-\r)--(\w*3.5,\h*.9)--(\w*2.75,\h*.9)--(\w*2.75,\h*.5+\r);
\draw[blue] (\w*1.75,\h*1.5-\r)--(\w*1.75,\h*.9)--(\w*.5,\h*.9)--(\w*.5,\h*.5+\r);

\end{scope}
\end{tikzpicture}
}
\end{center}
    To complete the pairings for $r=2$, pair particles with label $\ell=2$ in row 2 to unpaired particles in row 1. Finally, for $r=1$, label all unlabelled particles in row 1 with ``1'', which completes the pairing procedure:

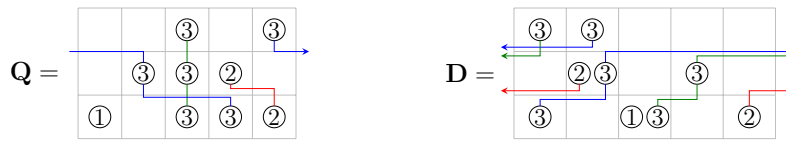
\begin{center}
\resizebox{!}{1.75cm}{
\begin{tikzpicture}[scale=0.7]
\def \w{1};
\def \h{1};
\def \r{0.25};
    
\begin{scope}[xshift=0cm]
\node at (-1,1.5) {\large $\Qbf=$};
\draw[gray!50,thin,step=\w] (0,0) grid (5*\w,3*\h);
\foreach \xx\yy\i\c in {2/2/3/white,4/2/3/white,1/1/3/white,2/1/3/white,3/1/2/white,0/0/1/white,2/0/3/white,3/0/3/white,4/0/2/white}
    {
    \draw[fill=\c](\w*.5+\w*\xx,\h*.5+\h*\yy) circle (\r cm);
    \node at (\w*.5+\w*\xx,\h*.5+\h*\yy) {\i};
    }

\draw[black!50!green] (\w*2.5,\h*2.5-\r)--(\w*2.5,\h*1.5+\r);
\draw[blue,-stealth] (\w*4.5,\h*2.5-\r)--(\w*4.5,\h*2.0)--(\w*5.3,\h*2.0);
\draw[blue] (-.2,\h*2.0)--(\w*1.5,\h*2.0)--(\w*1.5,\h*1.5+\r);
    
\draw[black!50!green] (\w*2.5,\h*1.5-\r)--(\w*2.5,\h*.5+\r);
\draw[blue] (\w*1.5,\h*1.5-\r)--(\w*1.5,\h*.95)--(\w*3.5,\h*.95)--(\w*3.5,\h*.5+\r);
\draw[red] (\w*3.5,\h*1.5-\r)--(\w*3.5,\h*1.15)--(\w*4.5,\h*1.15)--(\w*4.5,\h*.5+\r);

\end{scope}

\begin{scope}[xshift=10cm]
\def \w{1.2};
\node at (-1,1.5) {\large $\Dbf=$};
\draw[gray!50,thin,xstep=\w,ystep=\h] (0,0) grid (5*\w,3*\h);
\foreach \xx\yy\i\c in {0/2/3/white,1/2/3/white,.75/1/2/white,1.25/1/3/white,3/1/3/white,0/0/3/white,1.75/0/1/white,2.25/0/3/white,4/0/2/white}
    {
    \draw[fill=\c] (\w*.5+\w*\xx,\h*.5+\h*\yy) circle (\r cm);
    \node at (\w*.5+\w*\xx,\h*.5+\h*\yy) {\i};
    }

\draw[blue,-stealth] (\w*1.5,\h*2.5-\r)--(\w*1.5,\h*2.1)--(-.3,\h*2.1);
\draw[blue] (\w*5.2,\h*2)--(\w*1.75,\h*2)--(\w*1.75,\h*1.5+\r);
\draw[black!50!green,-stealth] (\w*.5,\h*2.5-\r)--(\w*.5,\h*1.9)--(-.3,\h*1.9);
\draw[black!50!green] (\w*5.2,\h*1.9)--(\w*3.5,\h*1.9)--(\w*3.5,\h*1.5+\r);
    
\draw[black!50!green] (\w*3.5,\h*1.5-\r)--(\w*3.5,\h*.9)--(\w*2.75,\h*.9)--(\w*2.75,\h*.5+\r);
\draw[blue] (\w*1.75,\h*1.5-\r)--(\w*1.75,\h*.9)--(\w*.5,\h*.9)--(\w*.5,\h*.5+\r);
\draw[red,-stealth] (\w*1.25,\h*1.5-\r)--(\w*1.25,\h*1.1)--(-.3,\h*1.1);
\draw[red] (\w*5.2,\h*1.1)--(\w*4.5,\h*1.1)--(\w*4.5,\h*.5+\r);

\end{scope}
\end{tikzpicture}
}
\end{center}

    The projection is read off the bottom rows to get $\Phi_{\FM}(\Qbf)=1\,0\,3\,3\,2$ and $\Phi_{\FM}(\Dbf)=(3,\emptyset,13,\emptyset,2)$.
\end{example}

The projection map $\Phi_{\FM}$ described in \cref{def:FM} gives an elegant solution to the balance equation~\eqref{eq:balance}.

\begin{theorem}[{\cite{FM07}}]\label{thm:FM}
Let $\lambda$ be a partition, and fix $n\in\mathbb{N}$. For a state $\bw \in \fermwords(\lambda', n)$ of the TASEP of type $(\lambda',n)$,
\[
\pr(\bw) = \frac{1}{\abs{\MLQ(\lambda,n)}} \Big\lvert \{\Qbf\in\MLQ(\lambda,n) \mid \Phi_{\FM}(\Qbf)=\bw\} \Big\rvert.
\]
\end{theorem}

We have an analogous theorem for the $0$-TAZRP in terms of the parameters $\xx=\{x_1,x_2,\ldots\}$. 

\begin{theorem}[{\cite{AMM22}}]\label{thm:tazrp}
Fix a partition $\lambda$ and $n\geq 1$. For a state $\bwt\in\bosonwords(\lambda',n)$ of the $0$-TAZRP of type $(\lambda',n)$,
\begin{equation}\label{eq:stationary}
\widetilde{\pr}(\bwt)=\frac{1}{\mathcal{Z}_{\lambda,n}} \sum_{\substack{\Dbf\in\bMLQ(\lambda,n)\\\widetilde{\Phi}_{\FM}(\Dbf)=\bwt}} \xx^{\Dbf},
\qquad \mbox{where} \qquad
\mathcal{Z}_{\lambda,n} = \sum_{\substack{\Dbf\in\bMLQ(\lambda,n)}}\xx^{\Dbf}.
\end{equation}
\end{theorem}

The solution to the $0$-TASEP balance equation can also be derived by specializing a more general object introduced in~\cite{AMM22} to the $t=0$ case, and it turns out that this construction coincides with the bosonic projection map $\Phi_{\FM} \colon \bMLQ(\lambda',n) \to \bosonwords(\lambda,n)$.

\begin{remark}
The kTAZRP process is also a Markov chain on $\bosonwords(\lambda,n)$. For a state $\bwt = \wt_1\cdots \wt_n$, there is a transition with rate $1$ to the state $\bwt' = \wt'_1\cdots \wt'_n$ given by
\[
\wt'_i=\begin{cases}
    \wt_j\setminus Y& \text{if } i=j,\\
    \wt_{j+1}\cup Y& \text{if } i=j+1,\\
    \wt_i& \text{if } i\neq j,j+1,
\end{cases}
\]
for some $1 \leq j \leq n$ and subset $\emptyset \neq Y \subseteq \wt_j$ such that $\min Y\geq \max \wt_j\setminus Y$.
Transitions between any other two states have rate $0$.
In other words, a valid transition is given by any number of the highest labelled particles from site $j$ hopping to site $j+1$, and every such transition occurs with uniform rate. Despite the different construction, the stationary distribution of kTAZRP, shown by the analog of \cref{thm:tazrp} in~\cite{KMO16}, is the same as that of the 0-TAZRP with rate parameters specialized at $x_1 = \cdots = x_n = 1$.
\end{remark}

\subsection{Kirillov--Reshetikhin crystals and the combinatorial R matrix}

For $r \in \{1, \dotsc, n\}$ and $s \in \ZZ_{>0}$, a \dfn{Kirillov--Reshetikhin (KR) crystal} $\calB^{r,s}$ (of affine type $A_{n-1}^{(1)}$ or of $\asl_n$) is the set of semistandard tableaux whose shape is an $r \times s$ rectangle, together with certain operators known as Kashiwara (crystal) operators.
Since we will not use the Kashiwara operators, we will omit the precise definition and instead refer the reader to~\cite{BS17,KKMMNN91,Shimozono02} for more information.
We denote the single column KR crystals $\calB^r := \calB^{r,1}$ and equate the elements with subsets of $[n]$ of size~$r$. Similarly, denote the single row KR crystals $\wcalB^s := \calB^{1,s}$ and equate the elements with elements of $\mathcal{M}(n)$ of size $s$. A tensor product of crystals is, set-theoretically, just a Cartesian product of crystals.
Hence, for any $\alpha = (\alpha_1, \dotsc, \alpha_k)$, we can identify
\[
\MLQ(\alpha, n) \longleftrightarrow \calB^{\alpha} := \calB^{\alpha_1} \otimes \cdots \otimes \calB^{\alpha_k}
\qquad \text{ and } \qquad
\bMLQ(\alpha, n) \longleftrightarrow \wcalB^{\alpha} := \wcalB^{\alpha_1} \otimes \cdots \otimes \wcalB^{\alpha_k}.
\]
That is fermionic (resp.\ bosonic) multiline queues are tensor products of KR crystals of the form $\calB^r$ (resp.\ $\wcalB^s$).

The \dfn{combinatorial R matrix} is the unique crystal isomorphism $\combR \colon \calB^{r,s} \otimes \calB^{r',s'} \to \calB^{r',s'} \otimes \calB^{r,s}$, which has an explicit-but-complicated algorithm for computing it for $\asl_n$ (see, \textit{e.g.},~\cite{Shimozono02}).
However, if either $s = s' = 1$ or $r = r' = 1$, then there is an easier description known as the \dfn{Nakayashiki--Yamada (NY) rule}~\cite[Rule~3.10, Rule~3.11]{NY95}.
If $s = s' = 1$ (resp.~$r = r' = 1$) the NY rule is precisely the involution $\sigma_1$ on $(\wcalB^r, \wcalB^{r'})$ (resp.~$(\wcalB^s, \wcalB^{s'})$) considered as a multiline queue with two rows.

The combinatorial R matrix is an action of the symmetric group on a tensor product of KR crystals~\cite{KKMMNN91}, which can also be seen as coming from Howe duality (for $asl_n$ only).
To make this precise, consider $\calB^{r_1,s_1} \otimes \cdots \otimes \calB^{r_k,s_k}$ with $\combR_i$ acting on $\calB^{r_i,s_i} \otimes \calB^{r_{i+1},s_{i+1}}$ (and fixing all of the other factors), or notationally, $\combR_i = 1^{\otimes i-1} \otimes \combR \otimes 1^{\otimes n-i-1}$. Thus the following relations hold:
\begin{itemize}
\item[i.] $\combR_i^2=1$,
\item[ii.] $\combR_i \circ \combR_j = \combR_j \circ \combR_i$ if $\abs{i-j} > 1$, and
\item[iii.] $\combR_i \circ \combR_{i+1} \circ \combR_i = \combR_{i+1} \circ \combR_i \circ \combR_{i+1}$.
\end{itemize}
Pictorially, the action of the combinatorial R matrix on the tensor $b \otimes c$ is represented as
\[
    \begin{tikzpicture}[scale=1.2,baseline=-10]
    \node (A) at (0,0) {$c'$};
    \node (B) at (0,-.6) {$b'$};
    \node (C) at (1,0) {$b$};
    \node (D) at (1,-.6) {$c$};
    \draw (A) -- (D);
    \draw (B) -- (C);
    \end{tikzpicture}
\qquad \text{ or } \qquad
    \begin{tikzpicture}[scale=.5,baseline=0]
    \node (A) at (0,1) {$b$};
    \node (C) at (1,0) {$c$};
    \node (B) at (-1,0) {$c'$};
    \node (D) at (0,-1) {$b'$};
    \draw (A) -- (D);
    \draw (B) -- (C);
    \end{tikzpicture},
\]
where $\combR(b \otimes c) = c' \otimes b'$.
Likewise, the braid relation $\combR_i \circ \combR_{i+1} \circ \combR_i = \combR_{i+1} \circ \combR_{i} \circ \combR_{i+1}$ is represented pictorially as
\[
\begin{tikzpicture}[scale=1,baseline=0]
    \def \h{.6};
    \node (A) at (0,0) {$c'$};
    \node (B) at (0,-1*\h) {$b'$};
    \node (C) at (0,-2*\h) {$a'$};
    \node (D) at (3,0) {$a$};
    \node (E) at (3,-1*\h) {$b$};
    \node (F) at (3,-2*\h) {$c$};
    \draw (A) -- (1,-1*\h)--(2,-2*\h)--(F);
    \draw (B) -- (1,0*\h)--(2,0*\h)--(E);
    \draw (C)--(1,-2*\h)--(2,-1*\h)--(D);
    \node at (4.2,-1*\h) {$=$};
\end{tikzpicture}
\qquad
\begin{tikzpicture}[scale=1,baseline=0]
    \def \h{.6};
    \node (A) at (0,0) {$c'$};
    \node (B) at (0,-1*\h) {$b'$};
    \node (C) at (0,-2*\h) {$a'$};
    \node (D) at (3,0) {$a$};
    \node (E) at (3,-1*\h) {$b$};
    \node (F) at (3,-2*\h) {$c$};
    \draw (A) -- (1,0)--(2,-1*\h)--(F);
    \draw (B) -- (1,-2*\h)--(2,-2*\h)--(E);
    \draw (C)--(1,-1*\h)--(2,0)--(D);
\end{tikzpicture}.
\]
Note that this might be different than conventions used in other papers, \textit{e.g.},~\cite{KMO15,KMO16II}.
The braid relation is also known as the \dfn{Yang--Baxter equation}, which can be written (when restricted to three factors) as
\[
(1 \otimes \combR) \circ (\combR \otimes 1) \circ(1 \otimes \combR) = (\combR \otimes 1) \circ (1 \otimes \combR) \circ (\combR \otimes 1).
\]

Let $\Qbf$ be a fermionic or bosonic multiline queue corresponding to an element in $\bb = b_1 \otimes \cdots \otimes b_k$ in $\calB^{\alpha}$ or $\wcalB^{\alpha}$, respectively. 
The action of the operator $\sigma_i$ on $\Qbf$ corresponds to the application of the fermionic NY rule~\cite[Rule~3.10]{NY95} or bosonic NY rule~\cite[Rule~3.11]{NY95}, respectively, to $b_i\otimes b_{i+1}$.
Hence, this corresponds to the combinatorial R matrix $\combR_i$, or in other words the diagrams
\[
    \begin{tikzcd}
    \MLQ(\alpha) \arrow[r,"\iota"] \arrow[d,"\sigma_i"'] & \calB^{\alpha} \arrow[d,"\combR_i"] \\
    \MLQ(s_i\alpha) \arrow[r,"\iota"] & \calB^{s_i\alpha}
    \end{tikzcd}
    \qquad\text{ and }\qquad
    \begin{tikzcd}
    \MLD(\alpha) \arrow[r,"\iota"] \arrow[d,"\sigma_i"'] & \wcalB^{\alpha} \arrow[d,"\combR_i"] \\
    \MLD(s_i\alpha) \arrow[r,"\iota"] & \wcalB^{s_i\alpha}
    \end{tikzcd}
\]
commute, where $s_i$ denotes the simple transposition of $i \leftrightarrow i+1$.
Consequently, we identify $\sigma_i$ with $\combR_i$ in parallel with considering multiline queues as elements in a tensor product of KR crystals.

\subsection{A Markov chain on straight (fermionic) multiline queues that projects to the TASEP}\label{sec:straight MC}

Ferrari and Martin introduced \dfn{ringing paths} in~\cite{FM07} to describe a uniformly distributed Markov process on (straight) fermionic multiline queues that projects (see \cref{def:projection} below for the precise meaning) to the TASEP. For a fixed $(\lambda,n)$, we define functions
\[
F,\R \colon \MLQ(\lambda,n)\times \ZZ_n\rightarrow \MLQ(\lambda,n)\times\ZZ_n
\]
where $F(\Qbf,i)$ gives a \dfn{forward transition} that aligns with the dynamics of the TASEP, and $\R(\Qbf,i)$ gives a \dfn{reverse transition} corresponding to the \dfn{time reversal process}.

\begin{defn}
\label{def:tasep ringing}
    Let $\lambda$ be a partition with $L:=\lambda_1$. For a multiline queue $\Qbf=(Q_1,\ldots,Q_L)\in\MLQ(\lambda,n)$ and a site number $i\in\ZZ_n$, a \dfn{(forward) ringing path} in $\Qbf$ at site $i$ is a sequence of sites $a_1,\ldots,a_L,a_{L+1}$, where $a_1=i$ and $a_j$ for $2\leq j\leq L+1$ is defined recursively:
    \[
    a_j=\begin{cases}
        a_{j-1},&\text{if } a_{j-1}\in Q_{j-1}, \\
        a_{j-1}+1,&\text{otherwise.}
    \end{cases}
    \]
    Then $F(\Qbf,i)=(\Qbf',a_{L+1})$, where $\Qbf'$ is the multiline queue $\Qbf$ with each instance of $(0,1)$ at sites $(a_{j}-1,a_j)$ sent to $(1,0)$ at those sites. In other words, for $1\leq j\leq L$, if there is a particle at site $a_j$ of row $j$ of $\Qbf$ with a vacancy at the site to its left, that particle hops to site $a_{j}-1$. The \emph{rate} of the transition in the forward Markov process from $\Qbf$ to $\Qbf'$ is 1 (uniform).
\end{defn}

The reverse ringing path is defined so as to act as an inverse for the forward ringing path, thus representing TASEP dynamics in reverse time.
\begin{defn}
\label{def:tasep reverse ringing}
    For a multiline queue $\Qbf=(Q_1, \ldots, Q_L) \in \MLQ(\lambda,n)$ and a site number $i\in\ZZ_n$, a \dfn{reverse ringing path} in $\Qbf$ at site $i$ is a sequence of sites $c_L,\ldots,c_1,c_{0}$ where $c_L=i$ and for $0\leq j\leq L-1$, $c_j$ is defined recursively:
    \[
    c_j=\begin{cases}
        c_{j+1},&\text{if } c_{j+1}+1\in Q_{j-1}\\
        c_{j+1}-1,&\text{otherwise}.
    \end{cases}
    \]
    Then $\R(\Qbf,i)=(\Qbf'',c_{0})$, where $\Qbf''$ is the multiline queue $\Qbf$ with each instance of $(1,0)$ at sites $(c_{j},c_j+1)$ sent to $(0,1)$ at those sites. In other words, for $1\leq j\leq L$, if there is a particle at site $c_j$ of row $j$ of $\Qbf$, that particle hops to site $c_{j}+1$. The \emph{rate} of the transition in the time reversed Markov process from $\Qbf$ to $\Qbf''$ is 1 (uniform).
\end{defn}

By comparing the sequences defining the forward and reverse ringing paths, one sees that $F$ and $\R$ are mutual inverses.
\begin{lemma}\label{lem:inv}
    For $\Qbf\in\MLQ(\lambda,n)$ and $1\leq i\leq n$, we have 
    $\R\bigl(F(\Qbf,i)\bigr) = F\bigl(\R(\Qbf,i)\bigr) = (\Qbf,i)$.
    \end{lemma}

From the above fact that each forward ringing path corresponds to a unique reverse ringing path, there is an equal number of incoming transitions as outgoing transitions from each multiline queue. Since all transitions occur with a uniform rate, \eqref{eq:balance} is satisfied by the uniform distribution on $\MLQ(\lambda,n)$.

Now we define the notion of a projection of Markov chains.

\begin{defn}[Projection of Markov chains]\label{def:projection}
    Let $\mathcal{A}$ and $\mathcal{B}$ be two (continuous) Markov chains with state spaces $S_{\mathcal{A}}$ and $S_{\mathcal{B}}$ and transition rates $\rate_\mathcal{A}\colon S_\mathcal{A}\times S_\mathcal{A} \rightarrow \RR_{>0}$ and $\rate_\mathcal{B}\colon S_\mathcal{B}\times S_\mathcal{B}\rightarrow \RR_{>0}$, respectively. Let $\pi \colon S_\mathcal{A}\rightarrow S_\mathcal{B}$ be a surjective map (projection). Then we say that the Markov chain $\mathcal{A}$ \dfn{projects} to $\mathcal{B}$ if for any $b,b'\in S_\mathcal{B}$, for all $a\in S_\mathcal{A}$ such that $\pi(a)=b$ we have
    \begin{equation}\label{eq:projection}
    \rate_\mathcal{B}(b,b')=\sum_{\substack{a'\in S_\mathcal{A}\\\pi(a')=b'}} \rate_\mathcal{A}(a,a').
    \end{equation}
\end{defn}

If $\mathcal{A}$ projects to $\mathcal{B}$, the unnormalized stationary distribution $\pr_\mathcal{B}$ of $\mathcal{B}$ is given by
\[
\pr_{\mathcal{B}}(b) = \sum_{\substack{x\in S_\mathcal{A}\\\pi(a)=b}} \pr_{\mathcal{A}}(a),
\]
where $\pr_{\mathcal{A}}$ is the (unnormalized) stationary distribution of $\mathcal{A}$.

It is proved in~\cite{FM07} that for every transition from $\Qbf$ to $\Qbf'$ in the MLQ chain such that $\Phi_{\FM}(\Qbf)\neq\Phi_{\FM}(\Qbf')$, there is a corresponding transition with the same rate from $\Phi_{\FM}(\Qbf)$ to $\Phi_{\FM}(\Qbf')$ in the TASEP chain. Conversely, for every $w,w'\in\fermwords$ such that $\rate(w,w')=1$, and for every $\Qbf\in\MLQ$ such that $\Phi_{\FM}(\Qbf)=w$, there is a unique $\Qbf'$ such that $\Phi_{\FM}(\Qbf')=w'$ and $\Qbf'=F_j(\Qbf)$ for some $1\leq j\leq n$, establishing the equality \eqref{eq:projection}. According to \cref{def:projection} this gives the following, which in turn leads to \cref{thm:FM}.

\begin{theorem}[{\cite[Thm.~4.1]{FM07}}]\label{thm:tasep projects FM}
    The Markov chain on $\MLQ(\lambda,n)$ defined by the forward ringing paths has the uniform distribution and projects to $TASEP$ of type $(\lambda',n)$ via the projection $\Phi_{\FM}$.
\end{theorem}

 The bosonic analogue of this section is examined in \cref{sec:bMLQ Markov}.

\section{Projection map on twisted multiline queues}\label{sec:projection map}

In this section, we define projection maps on fermionic twisted and bosonic multiline queues of shape $\alpha$, where $\alpha$ is any composition. When $\alpha$ is a partition, we recover the FM algorithm in \cref{def:FM}.

\subsection{Fermionic projection map}
For an integer $\ell\geq 1$ and $Q \subseteq [n]$, we define the operator $Q^{(\ell)} \colon \fermwords\rightarrow \fermwords$ is defined in terms of a pairing and labelling procedure.

\begin{defn}[Fermionic pairing procedure]\label{def:mlq pairing}
For $Q\subseteq [n]$ and an integer $\ell\geq 1$, the action of $Q^{(\ell)} \colon \fermwords\rightarrow \fermwords$ can be described as follows. Let $\uu\in\fermwords$ with $\ell<\min(\uu_+)$. The procedure is conducted in two phases: the \dfn{pairing phase} and the \dfn{collapsing phase}. 

\begin{itemize}[leftmargin=*]
\item To initialize the process, set $(y_1,\ldots,y_n)=(0,\ldots,0)$; this vector will record the labels assigned to paired and collapsed particles. Let $a=\min(\uu_+)$ and $k=\max(\uu_+)$. We process the particles in order from highest to lowest label, with all particles with label $r$ processed simultaneously, for $r=k,k-1,\ldots,a$. Let $u^{(r)}=\iota^{-1}(\nu_r(\uu))$ represent the set of particles in $u$ with label greater than or equal to $r$. 

\item \textbf{Pairing phase.} Let $s$ be maximal such that $\abs{u^{(s)}} \geq \abs{Q}$, if such exists. If it does not exist, set $s=a$. Sequentially, for $r=k,k-1,\ldots,s$, cylindrically (weakly to the right) pair the particles in $u^{(r)}$ 
to the particles in $Q$ according to $\Par(Q,u^{(r)})$, and assign to the \emph{newly paired} particles in $Q$ the label $r$. Specifically, when a particle in $Q$ at site $j$ is paired for the first time at step $r$, set  $y_j=r$.  

To complete the pairing phase, for each remaining unpaired particle in $Q$ in $\Par(Q,u^{(s)})$ set $y_j=\ell$, where $j$ is the site containing the particle.

\item \textbf{Collapsing phase.} Sequentially, for $r=s,s-1,\ldots,a$, for each \emph{newly unpaired} particle in $u^{(r)}$ in $\Par(Q,u^{(r)})$, collapse the particle to the row below and give it the label $r-1$. To record this, set $y_j = r-1$, where $j$ is the site in $u^{(r)}$ containing the unpaired particle.

\item Once all particles in $\uu$ have been processed, we set $Q^{(\ell)}(\uu)=y_1\cdots y_n\in\fermwords$.

\end{itemize}
\end{defn}

See \cref{ex:MLQ direct labelling} for an example of the pairing procedure.

In effect, on the multiline queue itself, we can think of the pairing phase as the usual FM pairing process (\cref{def:FM}) as long as there remain unpaired particles in the row below. Once all particles in the row below are paired to, the collapsing phase occurs. On the particle level (disregarding particle labels), this corresponds to applying the NY rule, where the collapsing particles are given the label associated to the corresponding step of the collapsing phase.

To see that $Q^{(\ell)}(\uu)=y_1\cdots y_n$ as obtained according to \cref{def:mlq pairing} is well-defined, it is enough to show that $y_j$ is assigned a value at most once for each $1\leq j\leq n$. Indeed, for each $j\in Q$, $y_j$ is assigned a label exactly once during the pairing phase. During the collapsing phase, if a particle at site $j$ collapses, necessarily $j\not\in Q$, so $y_j$ hasn't been assigned during the pairing phase. Since there is at most one particle per site in row $r$, each site in row $r-1$ receives at most one collapsing particle, and so $y_j$ is assigned a label at most once for $j\not\in Q$, as well.   

\begin{remark}\label{rem:direct labelling}
Observe that during the pairing phase at step $r$, the set of all particles in $B$ with the label $r$ is equal to 
the set of paired particles in $Q$ in $\Par(Q,u^{(r)})$ minus the set of paired particles in $Q$ in $\Par(Q,u^{(r+1)})$. During the collapsing phase at step $r$, the set of particles assigned the label $r-1$ is equal to the set of particles in $u^{(r)}$ that are unpaired in $\Par(Q,u^{(r)})$ minus the set of particles in $u^{(r)}$ that are unpaired in $\Par(D,u^{(r+1)})$. 
    
Consequently, the output $Q^{(\ell)}(\uu)$ of the procedure in \cref{def:mlq pairing} can be summarized as follows. Let let $v^{(r)}$ be the set of particles in $Q^{(\ell)}(\uu)$ that are assigned a label greater than or equal to $r$. 
    Then $v^{(r)}$ is given by 
    \[
        \left\{ \text{paired particles in}\ Q\ \text{in}\ \Par(Q,u^{(r)}) \right\} \; \cup \;
        \left\{ \text{unpaired particles in}\ u^{(r+1)}\ \text{in} \Par(Q,u^{(r+1)}) \right\}
    \]
    and $Q^{(\ell)}(\uu)=\iota(v^{(1)})+\cdots+\iota(v^{(k)})$. Alternatively, $Q^{(\ell)}(\uu)=y_1\cdots y_n$ where $y_j=\max\{i \mid j\in v^{(i)}\}$ if $j\in V^{(1)}$, and $y_j=0$ otherwise. Observe that $v^{(k)}\subseteq v^{(k-1)}\subseteq \cdots\subseteq v^{(1)}$, and so $\iota(v^{(k)})\leq\cdots\leq\iota(v^{(1)})$ is a sequence of nested indicator vectors. 

    In terms of the pairing/collapsing phases, $v^{(r)}$ is equal to:
    \begin{itemize}
    \item the set of paired particles $Q$ in $\Par(Q,u^{(r)})$ if $r$ is in the particle phase;
    \item $Q$ plus the set of unpaired particles in $u^{(r+1)}$ in $\Par(Q,u^{(r+1)})$ if $r+1$ is in the collapsing phase.
    \end{itemize}
\end{remark}

\begin{example}\label{ex:MLQ direct labelling}
    Let $\uu=3\,2\,4\,2\,5\,4\,3\,3\,0\,3$, $\ell=1$, and let $Q=\{2,3,4,7,9\}$. Then we have $a=2,k=5,s=3$, and $u^{(5)}\subseteq u^{(4)}\subseteq \cdots \subseteq u^{(1)}$ with
    \[
        u^{(5)}=\{5\},\quad u^{(4)}=\{3,5,6\},\quad u^{(3)}=\{1,3,5,6,7,8,10\},\quad
        u^{(2)}=u^{(1)}=\{1,2,3,4,5,6,7,8,10\}.
    \]
    To compute $v^{(k)},\ldots,v^{(1)}$ and $\yy=(y_1,\ldots,y_n)$, it suffices to record the sets of paired particles in $Q$ and unpaired particles in $u^{(r)}$ in $\Par(Q,u^{(r)})$ at each step, computing $v^{(r')}$ and updating $\yy$ at each step. Here $r$ is the row number and $r'$ is the label, which is equal to $r$ during the pairing phase and $r-1$ during the collapsing phase: 
    \begin{center}
\begin{tabular}{c|ccc|cc|c}
    &$r$&$\substack{\mbox{paired}\\\mbox{in}\  Q}$&$\substack{\mbox{unpaired}\\ \mbox{in}\ u^{(r)}}$&$r'$&$v^{(r')}$&$\yy$\\\hline
    $\substack{\text{pairing}\\\text{phase}}$&    5&$\{7\}$&$\emptyset$&5&$\{7\}$&$(0,0,0,0,0,0,5,0,0,0)$\\
       & 4&$\{3,7,9\}$&$\emptyset$&4&$\{3,7,9\}$&$(0,0,4,0,0,0,5,0,4,0)$\\
        &3&$\{2,3,4,7,9\}$&$\{5,6\}$&3&$\{2,3,4,7,9\}$&$(0,3,4,3,0,0,5,0,4,0)$\\\hline
$\substack{\text{collapsing}\\\text{phase}}$ &       3&$\{2,3,4,7,9\}$&$\{5,6\}$&2&$\{2,3,4,5,6,7,9\}$&$(0,3,4,3,2,2,5,0,4,0)$\\
      &2&$\{2,3,4,7,9\}$&$\{1,5,6,10\}$&1&$\{1,2,3,4,5,6,7,9,10\}$&$(1,3,4,3,2,2,5,0,4,1)$
    \end{tabular}
    \end{center}
 Thus we get
\[
Q^{(1)}(\uu) = \iota(v^{(1)}) + \cdots + \iota(v^{(5)}) = (1,3,4,3,2,2,5,0,4,1).
\]
so that $\iota(v^{(j)})=\nu_j(Q^{(1)}(\uu))$.

We also show each step diagrammatically. Unpaired particles in $Q$ are colored white, while unpaired particles in $u^{(r)}$ are colored red. During the pairing phase at step $r$, each newly paired particle in $Q$ (colored blue) acquires the label $r$. In the collapsing phase at step $r$, all newly unpaired particles in $u^{(r)}$ acquire the label $r'=r-1$.
\[
\begin{tikzpicture}[scale=0.7]
\def \w{0.6};
\def \h{.75};
\def \r{0.15};
    
\begin{scope}[xshift=0cm]
\node at (-.5,2.5*\h) {\large $u^{(5)}$};
\node at (-.5,1.5*\h) {\large $Q$};
\node at (-.5,.5*\h) {\large $\yy$};
\draw[gray!50,thin,xstep=\w,ystep=\h] (0,\h) grid (10*\w,3*\h);
\foreach \xx\yy\c in {4/2/black,
1/1/white,2/1/white,3/1/white,6/1/blue,8/1/white}
    {
    \draw[fill=\c] (\w*.5+\w*\xx,\h*.5+\h*\yy) circle (\r cm);
    }
\foreach \xx\i\c in {0/0/black,1/0/black,2/0/black,3/0/black,4/0/black,5/0/black,6/5/blue,7/0/black,8/0/black,9/0/black}    
{
    \node[\c] at (\w*.5+\w*\xx,\h*.5) {\i};
}
    
\draw[black!50!green,very thick] (\w*4.5,\h*2.5-\r)--(\w*4.5,\h*2)--(\w*6.5,\h*2)--(\w*6.5,\h*1.5+\r);

\end{scope}

\begin{scope}[xshift=8cm]
\node at (-.5,2.5*\h) {\large $u^{(4)}$};
\node at (-.5,1.5*\h) {\large $Q$};
\node at (-.5,.5*\h) {\large $\yy$};
\draw[gray!50,thin,xstep=\w,ystep=\h] (0,\h) grid (10*\w,3*\h);
\foreach \xx\yy\c in {2/2/black,4/2/black,5/2/black,
1/1/white,2/1/blue,3/1/white,6/1/black,8/1/blue}
    {
    \draw[fill=\c] (\w*.5+\w*\xx,\h*.5+\h*\yy) circle (\r cm);
    }
\foreach \xx\i\c in {0/0/black,1/0/black,2/4/blue,3/0/black,4/0/black,5/0/black,6/5/black,7/0/black,8/4/blue,9/0/black}    
{
    \node[\c] at (\w*.5+\w*\xx,\h*.5) {\i};
}
    
\draw[black!50!green,very thick] (\w*4.5,\h*2.5-\r)--(\w*4.5,\h*2.1)--(\w*8.5,\h*2.1)--(\w*8.5,\h*1.5+\r);
\draw[black!50!green,very thick] (\w*2.5,\h*2.5-\r)--(\w*2.5,\h*1.5+\r);
\draw[black!50!green,very thick] (\w*5.5,\h*2.5-\r)--(\w*5.5,\h*1.9)--(\w*6.5,\h*1.9)--(\w*6.5,\h*1.5+\r);

\end{scope}

\begin{scope}[yshift=0cm,xshift=16cm]
\node at (-.5,2.5*\h) {\large $u^{(3)}$};
\node at (-.5,1.5*\h) {\large $Q$};
\node at (-.5,.5*\h) {\large $\yy$};
\draw[gray!50,thin,xstep=\w,ystep=\h] (0,\h) grid (10*\w,3*\h);
\foreach \xx\yy\c in {0/2/black,2/2/black,4/2/red,5/2/red,6/2/black,7/2/black,9/2/black,
1/1/blue,2/1/black,3/1/blue,6/1/black,8/1/black}
    {
    \draw[fill=\c] (\w*.5+\w*\xx,\h*.5+\h*\yy) circle (\r cm);
    }
\foreach \xx\i\c in {0/0/black,1/3/blue,2/4/black,3/3/blue,4/0/black,5/0/black,6/5/black,7/0/black,8/4/black,9/0/black}    
{
    \node[\c] at (\w*.5+\w*\xx,\h*.5) {\i};
}
\draw[black!50!green,very thick] (\w*.5,\h*2.5-\r)--(\w*.5,\h*2.1)--(\w*1.5,\h*2.1)--(\w*1.5,\h*1.5+\r);
\draw[black!50!green,very thick] (\w*2.5,\h*2.5-\r)--(\w*2.5,\h*1.5+\r);
\draw[black!50!green,very thick] (\w*6.5,\h*2.5-\r)--(\w*6.5,\h*1.5+\r);
\draw[black!50!green,very thick] (\w*7.5,\h*2.5-\r)--(\w*7.5,\h*2)--(\w*8.5,\h*2)--(\w*8.5,\h*1.5+\r);
\draw[black!50!green,very thick,-stealth] (\w*9.5,\h*2.5-\r)--(\w*9.5,\h*1.9)--(\w*10.3,\h*1.9);
\draw[black!50!green,very thick] (-.2,\h*1.9)--(\w*3.5,\h*1.9)--(\w*3.5,\h*1.5+\r);

\end{scope}

\node at (1,3) {Pairing Phase};
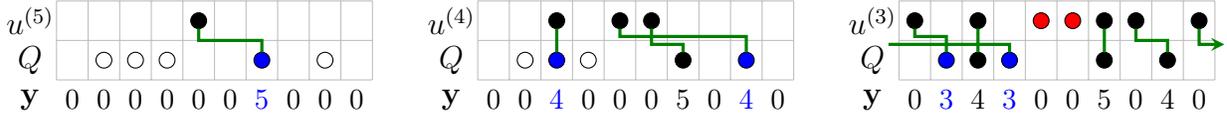
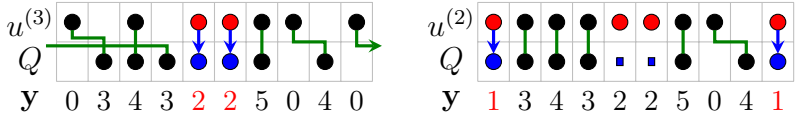
\node at (1,-1) {Collapsing Phase};

\begin{scope}[yshift=-4cm,xshift=0cm]
\node at (-.5,2.5*\h) {\large $u^{(3)}$};
\node at (-.5,1.5*\h) {\large $Q$};
\node at (-.5,.5*\h) {\large $\yy$};
\draw[gray!50,thin,xstep=\w,ystep=\h] (0,\h) grid (10*\w,3*\h);
\foreach \xx\yy\c in {0/2/black,2/2/black,9/2/black,4/2/red,5/2/red,6/2/black,7/2/black,4/1/blue,5/1/blue,
1/1/black,2/1/black,3/1/black,6/1/black,8/1/black}
    {
    \draw[fill=\c] (\w*.5+\w*\xx,\h*.5+\h*\yy) circle (\r cm);
    }
\foreach \xx\i\c in {0/0/black,1/3/black,2/4/black,3/3/black,4/2/red,5/2/red,6/5/black,7/0/black,8/4/black,9/0/black} 
{
    \node[\c] at (\w*.5+\w*\xx,\h*.5) {\i};
}

\draw[black!50!green,very thick] (\w*.5,\h*2.5-\r)--(\w*.5,\h*2.1)--(\w*1.5,\h*2.1)--(\w*1.5,\h*1.5+\r);
\draw[black!50!green,very thick] (\w*2.5,\h*2.5-\r)--(\w*2.5,\h*1.5+\r);
\draw[black!50!green,very thick] (\w*6.5,\h*2.5-\r)--(\w*6.5,\h*1.5+\r);
\draw[black!50!green,very thick] (\w*7.5,\h*2.5-\r)--(\w*7.5,\h*2)--(\w*8.5,\h*2)--(\w*8.5,\h*1.5+\r);
\draw[black!50!green,very thick,-stealth] (\w*9.5,\h*2.5-\r)--(\w*9.5,\h*1.9)--(\w*10.3,\h*1.9);
\draw[black!50!green,very thick] (-.2,\h*1.9)--(\w*3.5,\h*1.9)--(\w*3.5,\h*1.5+\r);

\draw[blue,very thick,-stealth] (\w*4.5,\h*2.5-\r)--(\w*4.5,\h*1.5+\r);
\draw[blue,very thick,-stealth] (\w*5.5,\h*2.5-\r)--(\w*5.5,\h*1.5+\r);
\end{scope}

\begin{scope}[yshift=-4cm,xshift=8cm]
\node at (-.5,2.5*\h) {\large $u^{(2)}$};
\node at (-.5,1.5*\h) {\large $Q$};
\node at (-.5,.5*\h) {\large $\yy$};
\draw[gray!50,thin,xstep=\w,ystep=\h] (0,\h) grid (10*\w,3*\h);
\foreach \xx\yy\c in {0/2/red,1/2/black,2/2/black,3/2/black,4/2/red,5/2/red,6/2/black,7/2/black,9/2/red,0/1/blue,9/1/blue,
1/1/black,2/1/black,3/1/black,6/1/black,8/1/black}
    {
    \draw[fill=\c] (\w*.5+\w*\xx,\h*.5+\h*\yy) circle (\r cm);
    }
\foreach \xx\i\c in {0/1/red,1/3/black,2/4/black,3/3/black,4/2/black,5/2/black,6/5/black,7/0/black,8/4/black,9/1/red}    
{
    \node[\c] at (\w*.5+\w*\xx,\h*.5) {\i};
}
\draw[fill=blue] (\w*4.4,\h*1.4) rectangle (\w*4.6,\h*1.6);
\draw[fill=blue] (\w*5.4,\h*1.4) rectangle (\w*5.6,\h*1.6);

\draw[black!50!green,very thick] (\w*1.5,\h*2.5-\r)--(\w*1.5,\h*1.5+\r);
\draw[black!50!green,very thick] (\w*2.5,\h*2.5-\r)--(\w*2.5,\h*1.5+\r);
\draw[black!50!green,very thick] (\w*3.5,\h*2.5-\r)--(\w*3.5,\h*1.5+\r);
\draw[black!50!green,very thick] (\w*6.5,\h*2.5-\r)--(\w*6.5,\h*1.5+\r);
\draw[black!50!green,very thick] (\w*7.5,\h*2.5-\r)--(\w*7.5,\h*2)--(\w*8.5,\h*2)--(\w*8.5,\h*1.5+\r);

\draw[blue,very thick,-stealth] (\w*0.5,\h*2.5-\r)--(\w*0.5,\h*1.5+\r);
\draw[blue,very thick,-stealth] (\w*9.5,\h*2.5-\r)--(\w*9.5,\h*1.5+\r);

\end{scope}
\end{tikzpicture}
\]
\end{example}

From here on, we will equate subsets $Q \subseteq[n]$ with the corresponding operators $Q^{(\ell)}$ acting on fermionic words, when $\ell$ is clear. We use these operators to define the projection map $\Phi\colon \MLQ\rightarrow\fermwords$ on twisted multiline queues. 

For a fermionic multiline queue $\Qbf=(Q_1,\ldots,Q_k)$, define $\Phi \colon \MLQ \rightarrow \fermwords$ by
\begin{equation}\label{eq:Phi def}
\Phi(\Qbf)\coloneqq Q_1(\cdots Q_k(\mathbf{w}_0)\cdots),
\end{equation}
where $\mathbf{w}_0=(0,\ldots,0)\in\fermwords$ and $Q_j\coloneqq Q_j^{(j)}$ for $j = 1,\dotsc,k$. Note that this composition of operators is well-defined since for each $j=k,\ldots,2,1$, the intermediate step $Q_j^{(j)}(\cdots Q_k^{(k)}(\mathbf{w}_0)\cdots)$ is a fermionic word in the letters $\{j,\ldots,k\}$.

See \cref{ex:MLQ labelling procedure} below for an example of $\Phi$.

In \cref{def:mlq pairing particlewise}, we give an equivalent description of the pairing procedure in which one particle is paired at a time, instead of simultaneously pairing all particles of the same label. This allows for the description of the labelling procedure as a probabilistic queueing process. For a straight multiline queue, that procedure coincides with the standard description of the FM pairing algorithm.

\begin{lemma}\label{lem:straight MLQ}
If $\Qbf$ is a straight fermionic multiline queue, the fermionic pairing procedure coincides with the FM algorithm.
\end{lemma}

\begin{proof}
    At each row of a straight multiline queue, every particle is paired during the particle phase, and hence producing the same terminal output as the FM algorithm. (This also coincides with the Kuniba--Maruyama--Okado pairing process~\cite{KMO15}; see~\cref{sec:fermionic_combR} below.) As the process is executed from top to bottom, the equivalence holds for the labelling of every intermediate row as well.
\end{proof}

\begin{example}\label{ex:MLQ labelling procedure}
We show the labelling of the multiline queue $\Qbf$ in \cref{ex:sigma}, one row at a time:
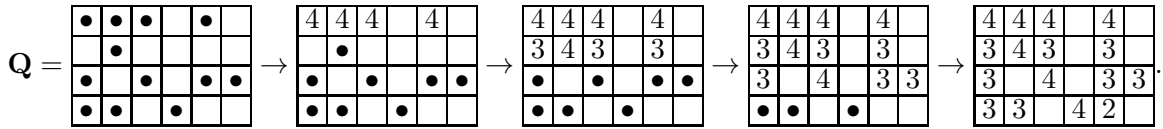
\[
\Qbf=\ytableaushort{{\bullet}{\bullet}{\bullet}{}{\bullet}{},{}{\bullet}{}{}{}{},{\bullet}{}{\bullet}{}{\bullet}{\bullet},{\bullet}{\bullet}{}{\bullet}{}{}}
\rightarrow
\ytableaushort{444{}4{},{}{\bullet}{}{}{}{},{\bullet}{}{\bullet}{}{\bullet}{\bullet},{\bullet}{\bullet}{}{\bullet}{}{}}
\rightarrow
\ytableaushort{444{}4{},343{}3{},{\bullet}{}{\bullet}{}{\bullet}{\bullet},{\bullet}{\bullet}{}{\bullet}{}{}}
\rightarrow
\ytableaushort{444{}4{},343{}3{},3{}4{}33,{\bullet}{\bullet}{}{\bullet}{}{}}
\rightarrow
\ytableaushort{444{}4{},343{}3{},3{}4{}33,33{}42{}}.
\]
Here is the straightening $\Qbf'=\sigma_2(\sigma_3(\sigma_1(\Qbf)))$ of $\Qbf$, and its labelling (where we verify it coincides with the FM algorithm):
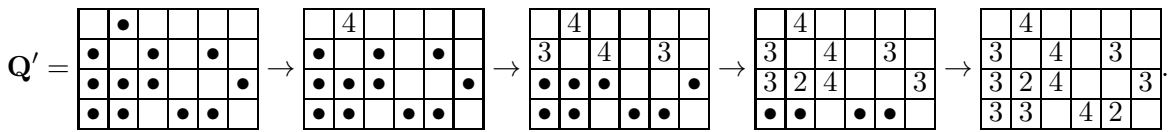
\[
\Qbf'=\ytableaushort{{}{\bullet}{}{}{}{},{\bullet}{}{\bullet}{}{\bullet}{},{\bullet}{\bullet}{\bullet}{}{}{\bullet},{\bullet}{\bullet}{}{\bullet}{\bullet}{}}
\rightarrow
\ytableaushort{{}4{}{}{}{},{\bullet}{}{\bullet}{}{\bullet}{},{\bullet}{\bullet}{\bullet}{}{}{\bullet},{\bullet}{\bullet}{}{\bullet}{\bullet}{}}
\rightarrow
\ytableaushort{{}4{}{}{}{},3{}4{}3{},{\bullet}{\bullet}{\bullet}{}{}{\bullet},{\bullet}{\bullet}{}{\bullet}{\bullet}{}}
\rightarrow
\ytableaushort{{}4{}{}{}{},3{}4{}3{},324{}{}3,{\bullet}{\bullet}{}{\bullet}{\bullet}{}}
\rightarrow
\ytableaushort{{}4{}{}{}{},3{}4{}3{},324{}{}3,33{}42{}}.
\]
In both cases, $\Phi(\Qbf)=\Phi(\Qbf')=3\,3\,0\,4\,2\,0$.
\end{example}

Our procedure is equivalent to that of~\cite{AAMP}.
See \cref{thm:equivalent_proj_AAMP} below for a precise statement.

\subsection{Bosonic projection map}\label{sec:bosonic projection}

The fermionic multiline queue labelling procedure naturally extends to a labelling procedure on bosonic multiline queues, modulo switching the direction in which particles are paired to be \emph{strictly to the left}, which is captured by the definition of $\bPar(A,B)$ for $A,B\in\bosonwords$. 

For an integer $\ell\geq 1$ and $D\in\mathcal{M}(n)$, the operator $D^{(\ell)} \colon \bosonwords\rightarrow \bosonwords$ is defined in terms of a labelling procedure that is analogous to that in \cref{def:mlq pairing}.

\begin{defn}[Bosonic pairing procedure]\label{def:mld pairing}
For $D\subseteq \mathcal{M}(n)$ and an integer $\ell\geq 1$, the action of $D^{(\ell)} \colon \bosonwords\rightarrow \bosonwords$ can be described as follows. Let $\widetilde{\mathbf{u}}=\widetilde{u}_1\cdots \widetilde{u}_n \in\bosonwords$ with $\ell<\min \bigcup_{i=1}^n \widetilde{u}_i$. The procedure is conducted in two phases: the \dfn{pairing phase} and the \dfn{collapsing phase}. 

\begin{itemize}[leftmargin=*]
\item To initialize the process, set $\mathbf{Y}=(Y_1,\ldots,Y_n)=(\emptyset,\ldots,\emptyset)$ to record the labels assigned to paired and collapsed particles. Let $a=\min \bigcup_{i=1}^n \widetilde{u}_i$ and $k=\max \bigcup_{i=1}^n \widetilde{u}_i$. We process the particles in order from highest to lowest label, with all particles with label $r$ processed simultaneously, for $r=k,k-1,\ldots,a$. Let $\widetilde{u}^{(r)}=\iota^{-1}\bigl(\nu_r(\widetilde{\mathbf{u}})\bigr)$ be the multiset of sites containing particles with label greater than or equal to $r$. 

\item \textbf{Pairing phase.} Let $s$ be maximal such that $|\widetilde{u}^{(s)}|\geq |D|$, if such exists. If it does not exist, set $s=a$. Sequentially, for $r=k,k-1,\ldots,s$, cylindrically (strictly to the left) pair the particles in $\widetilde{u}^{(r)}$ to the particles in $D$ according to $\bPar(D,\widetilde{u}^{(r)})$, and assign the \emph{newly paired} particles in $D$ the label $r$. To record this, when a particle in $D$ at site $j$ is paired for the first time at step $r$, add the element $r$ to $Y_j$. 

To complete the pairing phase, for each remaining unpaired particle in $D$ in $\bPar(D,\widetilde{u}^{(s)})$, add the element $\ell$ to $Y_j$, where $j$ is the site containing the particle. 

\item \textbf{Collapsing phase.} Sequentially, for $r=s,s-1,\ldots,a$, for each \emph{newly unpaired} particle in $\widetilde{u}^{(r)}$ in $\bPar(\widetilde{u}^{(r)},D)$, collapse the particle to the row below and give it the label $r-1$. To record this, add the element $r-1$ to $Y_j$, where $j$ is the site containing the unpaired particle in $\widetilde{u}^{(r)}$. 

\item Once all particles in $U$ have been processed, set $D^{(\ell)}(\widetilde{\mathbf{u}})=(Y_1,\ldots Y_n)\in\bosonwords$.
\end{itemize}
\end{defn}

See \cref{ex:MLD direct labelling} for an example of the procedure.

\begin{remark}\label{rem:direct labelling MLD}
During the pairing phase at step $r$, the set of particles given the label $r$ in $D$ is equal to the set of particles in $D$ paired in  $\bPar(D,\widetilde{u}^{(r)})$ minus the set of particles in $D$ paired in $\bPar(D,\widetilde{u}^{(r+1)})$. During the collapsing phase at step $r$, the set of particles assigned the label $r-1$ is precisely the set of particles in $\widetilde{u}^{(r)}$ that are unpaired in $\bPar(D,\widetilde{u}^{(r)})$ minus the set of particles in $\widetilde{u}^{(r)}$ that are unpaired in $\bPar(D,\widetilde{u}^{(r+1)})$. 

Consequently, the output $D^{(\ell)}(\bwt)$ of the procedure in \cref{def:mld pairing} can be summarized as follows.  Let $\widetilde{v}^{(r)}$ be the multiset of indices in $D^{(\ell)}(\bwt)$ that are assigned a label greater than or equal to $r$.  Then $\widetilde{v}^{(r)}$ is given by 
    \[
        \left\{ \text{paired particles in $D$ in} \bPar(D,\widetilde{u}^{(r)}) \right\} \; \cup \; 
        \left\{ \text{unpaired particles in $\widetilde{u}^{(r+1)}$ in}\bPar(D,\widetilde{u}^{(r+1)}) \right\},
    \]
and $D^{(\ell)}(\bwt)=\iota(\widetilde{v}^{(1)})\oplus\cdots\oplus\iota(\widetilde{v}^{(k)})$. Observe that $\widetilde{v}^{(k)}\subseteq\cdots\subseteq \widetilde{v}^{(1)}$, and thus $\iota(\widetilde{v}^{(k)})\leq\cdots\leq\iota(\widetilde{v}^{(1)})$ is a sequence of nested (bosonic) indicator vectors.

In terms of the pairing/collapsing phases, $\widetilde{v}^{(r)}$ is equal to:
\begin{itemize}
    \item the multiset of paired particles in $D$ in $\bPar(D,\widetilde{u}^{(r)})$ if $r$ is in the particle phase;
    \item the union of $D$ with the multiset of unpaired particles in $\widetilde{u}^{(r+1)}$ in $\bPar(D,\widetilde{u}^{(r+1)})$ if $r+1$ is in the collapsing phase.
\end{itemize}
\end{remark}

\begin{example}\label{ex:MLD direct labelling}
    Let $\widetilde{\mathbf{u}}=(446,\emptyset\,234,3,36)$ 
    and let $D=\{1,1,2,4\}$. Then we have $a=2,k=5,s=4$, and
    \begin{align*}
        \widetilde{u}^{(6)}=\widetilde{u}^{(5)}&=\{1,5\},&\qquad \widetilde{u}^{(4)}& =\{1,1,1,3,5\},\\
        \widetilde{u}^{(3)}&=\{1,1,1,3,3,4,5,5\},&\qquad \widetilde{u}^{(2)}=\widetilde{u}^{(1)}&=\{1,1,1,3,3,3,4,5,5\}.\\
    \end{align*}
    We show the corresponding sets of paired and unpaired particles in $\bPar(D,\widetilde{u}^{(r)})$ at each step to compute $\widetilde{v}^{(i)}$ and $\mathbf{Y}$.
    \begin{center}
\begin{tabular}{c|ccc|cc|c}
    &$r$&$\substack{\mbox{paired}\\\mbox{in $D$}}$&$\substack{\mbox{unpaired}\\\mbox{in $\widetilde{u}^{(r)}$}}$&$r'$&$\widetilde{v}^{(r')}$&$\mathbf{Y}$\\\hline
    $\substack{\mbox{pairing}\\\mbox{phase}}$& 6,5&$\{2,4\}$&$\emptyset$&6,5&$\{2,4\}$&$(\emptyset,6,\emptyset,6,\emptyset)$\\
       & 4&$\{1,1,2,4\}$&$\{1\}$&4&$\{1,1,2,4\}$&$(44,6,\emptyset,6,\emptyset)$\\\hline
$\substack{\mbox{collapsing}\\\mbox{phase}}$
        &4&$\{1,1,2,4\}$&$\{1\}$&3&$\{1,1,1,2,4\}$&$(344,6,\emptyset,6,\emptyset)$\\
        & 3&$\{1,1,2,4\}$&$\{1,1,1,5\}$&2&$\{1^5,2,4,5\}$&$(22344,6,\emptyset,6,2)$\\
      &2&$\{1,1,2,4\}$&$\{1,1,1,4,5\}$&1&$(\{1^5,2,4,4,5\}$&$(22344,6,\emptyset,16,2)$
    \end{tabular}
    \end{center}
    Thus we have $\iota(\widetilde{v}^{(j)})=\nu_j(D^{(1)}(\widetilde{\mathbf{u}}))$ for $j=1,\ldots,6$ with 
    \[D^{(1)}(\widetilde{\mathbf{u}})=\iota(\widetilde{v}^{(1)})\oplus\cdots\oplus \iota(\widetilde{v}^{(6)})=(22344,6,\emptyset,16,2). 
    \]
    We show each step diagrammatically, following the same conventions as in \cref{ex:MLQ direct labelling}. We skip the step $r=5$ since it is the same as $r=6$. 

\begin{align*}
&\begin{tikzpicture}[scale=0.7]
\def \w{1.8};
\def \h{.75};
\def \r{0.15};
\begin{scope}[xshift=0cm]
\node at (-.7,2.5*\h) {\large $U^{(6)}$};
\node at (-.7,1.5*\h) {\large $D$};
\node at (-.7,.5*\h) {\large $\mathbf{Y}$};
\draw[gray!50,thin,xstep=\w,ystep=\h] (0,\h) grid (5*\w,3*\h);
\foreach \xx\yy\c in {0/2/black,4/2/black,
-.2/1/white,.2/1/white,1/1/blue,3/1/blue}
{
    \draw[fill=\c] (\w*.5+\w*\xx,\h*.5+\h*\yy) circle (\r cm);
}
\foreach \xx\i in {0/$\emptyset$,1/$\tcb{6}$,2/$\emptyset$,3/$\tcb{6}$,4/$\emptyset$}    
{
    \node at (\w*.5+\w*\xx,\h*.5) {\i};
}
\draw[black!50!green,very thick] (\w*4.5,\h*2.5-\r)--(\w*4.5,\h*2.1)--(\w*3.5,\h*2.1)--(\w*3.5,\h*1.5+\r);
\draw[black!50!green,very thick,-stealth] (\w*0.5,\h*2.5-\r)--(\w*0.5,\h*1.9)--(-.3,\h*1.9);
\draw[black!50!green,very thick] (.3+5*\w,\h*1.9)--(\w*1.5,\h*1.9)--(\w*1.5,\h*1.5+\r);
\end{scope}
%
\begin{scope}[xshift=12cm]
\node at (-.7,2.5*\h) {\large $U^{(4)}$};
\node at (-.7,1.5*\h) {\large $D$};
\node at (-.7,.5*\h) {\large $\mathbf{Y}$};
\draw[gray!50,thin,xstep=\w,ystep=\h] (0,\h) grid (5*\w,3*\h);
\foreach \xx\yy\c in {{-.3}/2/black,0/2/black,.3/2/red, 2/2/black,4/2/black,
{-.2}/1/blue,.2/1/blue,1/1/black,3/1/black}
{
    \draw[fill=\c] (\w*.5+\w*\xx,\h*.5+\h*\yy) circle (\r cm);
}
\foreach \xx\i in {0/$\tcb{44}$,1/$6$,2/$\emptyset$,3/$6$,4/$\emptyset$}  
{
    \node at (\w*.5+\w*\xx,\h*.5) {\i};
}
\draw[black!50!green,very thick] (\w*4.5,\h*2.5-\r)--(\w*4.5,\h*2.1)--(\w*3.5,\h*2.1)--(\w*3.5,\h*1.5+\r);
\draw[black!50!green,very thick] (\w*2.5,\h*2.5-\r)--(\w*2.5,\h*2.1)--(\w*1.5,\h*2.1)--(\w*1.5,\h*1.5+\r);
\draw[black!50!green,very thick,-stealth] (\w*0.5,\h*2.5-\r)--(\w*0.5,\h*2)--(-.3,\h*2);
\draw[black!50!green,very thick] (.3+5*\w,\h*1.8)--(\w*0.3,\h*1.8)--(\w*0.3,\h*1.5+\r);
\draw[black!50!green,very thick,-stealth] (\w*0.2,\h*2.5-\r)--(\w*0.2,\h*1.9)--(-.3,\h*1.9);
\draw[black!50!green,very thick] (.3+5*\w,\h*1.9)--(\w*0.7,\h*1.9)--(\w*0.7,\h*1.5+\r);
\end{scope}
\node at (1,3) {{\bf Pairing Phase}};
\end{tikzpicture}
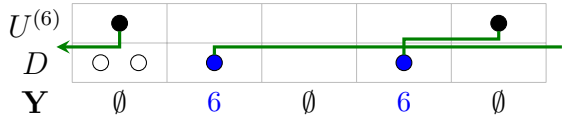
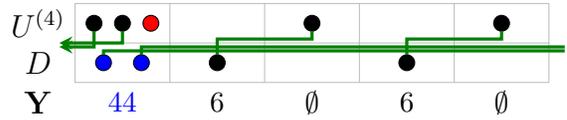
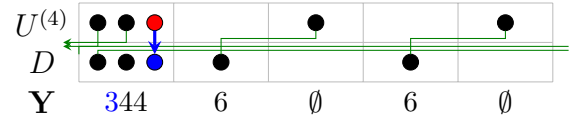
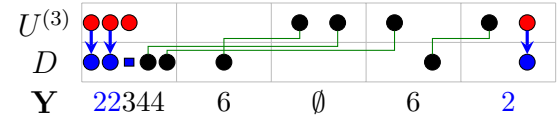
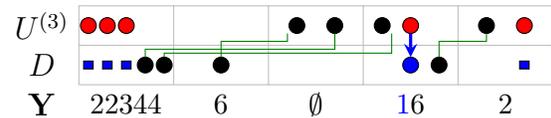
\allowdisplaybreaks \\
&\begin{tikzpicture}[scale=0.7]
\def \w{1.8};
\def \h{.75};
\def \r{0.15};
\node at (1,-1) {{\bf Collapsing Phase}};
%
\begin{scope}[yshift=-4cm,xshift=0cm]
\node at (-.7,2.5*\h) {\large $U^{(4)}$};
\node at (-.7,1.5*\h) {\large $D$};
\node at (-.7,.5*\h) {\large $\mathbf{Y}$};
\draw[gray!50,thin,xstep=\w,ystep=\h] (0,\h) grid (5*\w,3*\h);
\foreach \xx\yy\c in {-.3/2/black,0/2/black,.3/2/red, 2/2/black,4/2/black,
-.3/1/black,0/1/black,0.3/1/blue,1/1/black,3/1/black}
    {
    \draw[fill=\c] (\w*.5+\w*\xx,\h*.5+\h*\yy) circle (\r cm);
    }
\foreach \xx\i in {0/{$\tcb{3}44$},1/$6$,2/$\emptyset$,3/$6$,4/$\emptyset$}    
{
    \node at (\w*.5+\w*\xx,\h*.5) {\i};
}
\draw[black!50!green,thin] (\w*4.5,\h*2.5-\r)--(\w*4.5,\h*2.1)--(\w*3.5,\h*2.1)--(\w*3.5,\h*1.5+\r);
\draw[black!50!green,thin] (\w*2.5,\h*2.5-\r)--(\w*2.5,\h*2.1)--(\w*1.5,\h*2.1)--(\w*1.5,\h*1.5+\r);
\draw[black!50!green,thin,-stealth] (\w*0.5,\h*2.5-\r)--(\w*0.5,\h*2)--(-.3,\h*2);
\draw[black!50!green,thin] (.3+5*\w,\h*1.8)--(\w*0.2,\h*1.8)--(\w*0.2,\h*1.5+\r);
\draw[black!50!green,thin,-stealth] (\w*0.2,\h*2.5-\r)--(\w*0.2,\h*1.9)--(-.3,\h*1.9);
\draw[black!50!green,thin] (.3+5*\w,\h*1.9)--(\w*0,\h*1.9)--(\w*0,\h*1.5+\r);
\draw[blue,very thick,-stealth] (\w*0.8,\h*2.5-\r)--(\w*0.8,\h*1.5+\r);
\end{scope}
%
\begin{scope}[yshift=-4cm,xshift=12cm]
\node at (-.7,2.5*\h) {\large $U^{(3)}$};
\node at (-.7,1.5*\h) {\large $D$};
\node at (-.7,.5*\h) {\large $\mathbf{Y}$};
\draw[gray!50,thin,xstep=\w,ystep=\h] (0,\h) grid (5*\w,3*\h);
\foreach \xx\yy\c in {-.4/2/red,-.2/2/red,0/2/red, 1.8/2/black,2.2/2/black,2.8/2/black,3.8/2/black,4.2/2/red,
-.4/1/blue,-.2/1/blue,.2/1/black,.4/1/black,1/1/black,3.2/1/black,4.2/1/blue}
    {
    \draw[fill=\c] (\w*.5+\w*\xx,\h*.5+\h*\yy) circle (\r cm);
    }
\draw[fill=blue] (\w*.45,\h*1.4) rectangle (\w*.55,\h*1.6);
\foreach \xx\i in {0/{$\tcb{22}344$},1/$6$,2/$\emptyset$,3/$6$,4/$\tcb{2}$}    
{
    \node at (\w*.5+\w*\xx,\h*.5) {\i};
}
\draw[black!50!green,thin] (\w*4.3,\h*2.5-\r)--(\w*4.3,\h*2.1)--(\w*3.7,\h*2.1)--(\w*3.7,\h*1.5+\r);
\draw[black!50!green,thin] (\w*2.3,\h*2.5-\r)--(\w*2.3,\h*2.1)--(\w*1.5,\h*2.1)--(\w*1.5,\h*1.5+\r);
\draw[black!50!green,thin] (\w*2.7,\h*2.5-\r)--(\w*2.7,\h*1.9)--(\w*.7,\h*1.9)--(\w*.7,\h*1.5+\r);
\draw[black!50!green,thin] (\w*3.3,\h*2.5-\r)--(\w*3.3,\h*1.8)--(\w*.9,\h*1.8)--(\w*.9,\h*1.5+\r);
\draw[blue,very thick,-stealth] (\w*.1,\h*2.5-\r)--(\w*.1,\h*1.5+\r);
\draw[blue,very thick,-stealth] (\w*.3,\h*2.5-\r)--(\w*.3,\h*1.5+\r);
\draw[blue,very thick,-stealth] (\w*4.7,\h*2.5-\r)--(\w*4.7,\h*1.5+\r);
\end{scope}
%
\begin{scope}[yshift=-7cm,xshift=0cm]
\node at (-.7,2.5*\h) {\large $U^{(3)}$};
\node at (-.7,1.5*\h) {\large $D$};
\node at (-.7,.5*\h) {\large $\mathbf{Y}$};
\draw[gray!50,thin,xstep=\w,ystep=\h] (0,\h) grid (5*\w,3*\h);
\foreach \xx\yy\c in {-.4/2/red,-.2/2/red,0/2/red, 1.8/2/black,2.2/2/black,2.7/2/black,3.8/2/black,3/2/red,4.2/2/red,
.2/1/black,.4/1/black,1/1/black,3.3/1/black,3/1/blue}
    {
    \draw[fill=\c] (\w*.5+\w*\xx,\h*.5+\h*\yy) circle (\r cm);
    }
\draw[fill=blue] (\w*.05,\h*1.4) rectangle (\w*.15,\h*1.6);
\draw[fill=blue] (\w*.25,\h*1.4) rectangle (\w*.35,\h*1.6);
\draw[fill=blue] (\w*.45,\h*1.4) rectangle (\w*.55,\h*1.6);
\draw[fill=blue] (\w*4.65,\h*1.4) rectangle (\w*4.75,\h*1.6);
\foreach \xx\i in {0/$22344$,1/$6$,2/$\emptyset$,3/$\tcb{1}6$,4/$2$}    
{
    \node at (\w*.5+\w*\xx,\h*.5) {\i};
}
\draw[black!50!green,thin] (\w*4.3,\h*2.5-\r)--(\w*4.3,\h*2.1)--(\w*3.8,\h*2.1)--(\w*3.8,\h*1.5+\r);
\draw[black!50!green,thin] (\w*2.2,\h*2.5-\r)--(\w*2.2,\h*2.1)--(\w*1.5,\h*2.1)--(\w*1.5,\h*1.5+\r);
\draw[black!50!green,thin] (\w*2.7,\h*2.5-\r)--(\w*2.7,\h*1.9)--(\w*.7,\h*1.9)--(\w*.7,\h*1.5+\r);
\draw[black!50!green,thin] (\w*3.3,\h*2.5-\r)--(\w*3.3,\h*1.8)--(\w*.9,\h*1.8)--(\w*.9,\h*1.5+\r);
\draw[blue,very thick,-stealth] (\w*3.5,\h*2.5-\r)--(\w*3.5,\h*1.5+\r);
\end{scope}
\end{tikzpicture}
\end{align*}
\end{example}

The bosonic pairing procedure can be alternatively interpreted as pairing one particle at a time, in priority order from highest to lowest label. When the number of particles is at most $\abs{D}$, the procedure is identical to the bosonic analogue of the Ferrari--Martin algorithm described by Kuniba--Maruyama--Okado in~\cite{KMO16}. See \cref{def:mlq pairing particlewise} for the detailed description. 

Define $\widetilde{\Phi} \colon \bMLQ \rightarrow \bosonwords$ as follows. Let $\Dbf = (D_1,\ldots,D_k)$ be a bosonic multiline queue, and let $\widetilde{\mathbf{w}}_0=(\emptyset,\ldots,\emptyset)\in\bosonwords$. Then
\begin{equation}\label{eq:phi bosonic}
\widetilde{\Phi}(\Dbf)\coloneqq D_1(\cdots D_{k}(\widetilde{\mathbf{w}}_0)\cdots), 
\end{equation}
where we write $D_j\coloneqq D_j^{(j)}$ for $j=1,\ldots,k$.

The proof of the following is identical to that of \cref{lem:straight MLQ}.

\begin{lemma}
\label{lem:straight_boson_FM}
If $\Dbf$ is a straight bosonic multiline queue, the pairing procedure in \cref{def:mld pairing} coincides with the bosonic FM algorithm in~\cite[Sec.~4.4]{KMO16}.
\end{lemma}

\begin{example}\label{ex:MLD labelling procedure}
We show the labelling of the bosonic multiline queue $\Dbf$ from \cref{ex:sigma}.
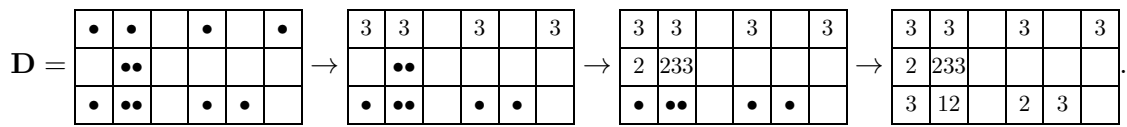
\[
\ytableausetup{boxsize=1.6em}
\Dbf=
\scalebox{.8}{$\ytableaushort{{\bullet}{\bullet}{}{\bullet}{}{\bullet},{}{\bullet\bullet}{}{}{}{},{\bullet}{\bullet\bullet}{}{\bullet}{\bullet}{}}$} \rightarrow
\scalebox{.8}{$\ytableaushort{{3}{3}{}{3}{}{3},{}{\bullet\bullet}{}{}{}{},{\bullet}{\bullet\bullet}{}{\bullet}{\bullet}{}}$} \rightarrow
\scalebox{.8}{$\ytableaushort{{3}{3}{}{3}{}{3},{2}{233}{}{}{}{},{\bullet}{\bullet\bullet}{}{\bullet}{\bullet}{}}$}\rightarrow
\scalebox{.8}{$\ytableaushort{{3}{3}{}{3}{}{3},{2}{233}{}{}{}{},{3}{12}{}{2}{3}{}}$}.
\]
We also show the straightening $\Dbf' = \sigma_2(\Dbf)$  and its labelling:
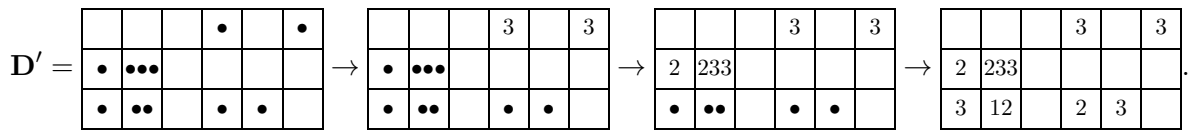
\[
\ytableausetup{boxsize=1.7em}
\mathbf{D'} = \scalebox{.8}{$\ytableaushort{{}{}{}{\bullet}{}{\bullet},{\bullet}{\bullet{\bullet}\bullet}{}{}{}{},{\bullet}{\bullet\bullet}{}{\bullet}{\bullet}{}}$}
\rightarrow
\scalebox{.8}{$\ytableaushort{{}{}{}{3}{}{3},{\bullet}{\bullet{\bullet}\bullet}{}{}{}{},{\bullet}{\bullet\bullet}{}{\bullet}{\bullet}{}}$}
\rightarrow
\scalebox{.8}{$\ytableaushort{{}{}{}{3}{}{3},{2}{233}{}{}{}{},{\bullet}{\bullet\bullet}{}{\bullet}{\bullet}{}}$}
\rightarrow
\scalebox{.8}{$\ytableaushort{{}{}{}{3}{}{3},{2}{233}{}{}{}{},{3}{12}{}{2}{3}{}}$}.
\]
In both cases, $\Phi(\Dbf)=\Phi(\mathbf{D'})=(3,12,\emptyset,2,3,\emptyset)$.
\end{example}

\subsection{Equivalence of iterative projection maps}\label{sec:GY}

In this section, we show that any two projection maps that are defined iteratively on a multiline queue $\Qbf=(Q_1,\ldots,Q_k)$, whose outputs are invariant of the action of the twist operator $\sigma_i$, and which coincide when $\Qbf$ is straight, are equivalent.

\begin{prop}\label{prop:projection maps equivalent}
    Suppose we have operators $\Psi^{(j)}_Q,\Omega^{(j)}_Q:\fermwords(n)\rightarrow \fermwords(n)$ for each $Q\subseteq [n]$ and $j\geq 1$. Let $\Psi,\Omega\colon \MLQ\rightarrow \fermwords$ be two projection maps defined by
    \[
    \Psi(\Qbf) = \Psi^{(1)}_{Q_1}\bigl(\Psi^{(2)}_{Q_2}(\cdots\Psi^{(k)}_{Q_k}(\uu_{k})\cdots)\bigr)
    \]
    and 
    \[
    \Omega(\Qbf) = \Omega^{(1)}_{Q_1}\bigl(\Omega^{(2)}_{Q_2}(\cdots\Omega^{(k)}_{Q_k}(\vv_{k})\cdots)\bigr)
    \]
    for any $\Qbf=(Q_1,\ldots,Q_k)\in\MLQ$, where $\uu_{0},\vv_{0}\in\fermwords$ are fixed, and $\uu_{j}=\uu_{0}^{(+j)}$ and $\vv_{j} = \vv_{0}^{(+j)}$ for $j\geq 1$. Suppose $\Psi$ and $\Omega$ satisfy the conditions
    \begin{itemize}
        \item[i.]  $\Psi(\Qbf)=\Omega(\Qbf)$ when $\Qbf$ is a straight multiline queue;
        \item[ii.] for any $i\geq 1$, $\Psi(\Qbf) = \Psi\bigl(\sigma_i(\Qbf)\bigr)$ and $\Omega(\Qbf) = \Omega\bigl(\sigma_i(\Qbf)\bigr)$; and
        \item[iii.] for $j\geq 1$, $Q\subseteq [n]$, and $\uu$ such that $\Psi^{(j)}_Q(\uu)$ and $\Omega^{(j)}_Q(\uu)$ are respectively well-defined, we have $\Psi^{(j+1)}_Q(\uu^{(+1)})=\Psi^{(j)}_Q(\uu)^{(+1)}$ and $\Omega^{(j+1)}_Q(\uu^{(+1)})=\Omega^{(j)}_Q(\uu)^{(+1)}$, respectively.
    \end{itemize}
    Then, $\Psi(\Qbf)=\Omega(\Qbf)$ for all $\Qbf\in\MLQ$, and 
    \[
    \Psi^{(j)}_{Q_j}\bigl(\Psi^{(j+1)}_{Q_{j+1}}(\cdots\Psi^{(k)}_{Q_k}(\uu_{k})\cdots)\bigr) = \Omega^{(j)}_{Q_j}\bigl(\Omega^{(j+1)}_{Q_{j+1}}(\cdots\Omega^{(k)}_{Q_k}(\vv_{k})\cdots)\bigr)
    \]
    for any intermediate step $1\leq j\leq k$, as well.
\end{prop}

\begin{proof}
    Let $\Qbf=(Q_1,\ldots,Q_k)\in\MLQ$. Conditions (i.) and (ii.) immediately imply that $\Psi(\Qbf)=\Omega(\Qbf)$. Fix $1<j\leq k$ and let $\mathbf{G}=(Q_j,\ldots,Q_k)$ be the restriction of $\Qbf$ to rows $j$ through $k$. Then we also have that
    \[
    \Psi(\mathbf{G}) = \Psi^{(1)}_{Q_j}\bigl(\Psi^{(2)}_{Q_{j+1}}(\cdots\Psi^{(k-j+1)}_{Q_k}(\uu_{k-j+1})\cdots)\bigr) = \Omega^{(1)}_{Q_j}\bigl(\Omega^{(2)}_{Q_{j+1}}(\cdots\Omega^{(k-j+1)}_{Q_k}(\vv_{k-j+1})\cdots)\bigr) = \Omega(\mathbf{G}).
    \]
    However, by iteratively applying condition 3, we have
    \begin{align*}
    \Psi^{(1)}_{Q_j}\bigl(\Psi^{(2)}_{Q_{j+1}}(\cdots\Psi^{(k-j+1)}_{Q_k}(\uu_{k-j+1})\cdots)\bigr)^{(+j-1)} &= \Psi^{(j)}_{Q_j}\bigl(\Psi^{(j+1)}_{Q_{j+1}}(\cdots\Psi^{(k)}_{Q_k}(\uu_{k})\cdots)\bigr),
    \\ 
    \Omega^{(1)}_{Q_j}\bigl(\Omega^{(2)}_{Q_{j+1}}(\cdots\Omega^{(k-j+1)}_{Q_k}(\vv_{k-j+1})\cdots)\bigr)^{(+j-1)} &= \Omega^{(j)}_{Q_j}\bigl(\Omega^{(j+1)}_{Q_{j+1}}(\cdots\Omega^{(k)}_{Q_k}(\vv_{k})\cdots)\bigr),
    \end{align*} 
    proving the desired equality.
\end{proof}

In~\cite{AAMP}, an alternative description of the Ferrari--Martin procedure was introduced using \dfn{graveyard diagrams} to generalize the procedure to fermionic twisted multiline queues. We can apply \cref{prop:projection maps equivalent} to show that our projection map is equivalent to the graveyard diagram formulation. Specifically: (1) the graveyard diagram procedure produces the same output as the FM procedure for a straight multiline queue~\cite[App.~1]{AGS20}, (2) the procedure satisfies condition (iii) since uniformly incrementing the input leaves the dynamics of the pairing procedure unchanged, and (3) the procedure is invariant under $\sigma_i$ as established in~\cite[Thm.~3.1]{AGS20}. In particular, we get the following equivalence:

\begin{cor}\label{thm:equivalent_proj_AAMP}
    Let $\Qbf=(Q_1,\ldots,Q_k)$ be a fermionic multiline queue and let $L_G(\Qbf)$ be its labelled graveyard diagram obtained by the procedure~\cite[Sec.~4.5]{AAMP}.  
    Then  
    \[
    Q_j^{(j)}\bigl(Q_{j+1}^{(j+1)}(\cdots Q_k^{(k)}(\mathbf{w}_k)\cdots)\bigr) = L_G(\Qbf)_j.
    \] 
    Moreover, if the sites with label $j-1$ in $L_G(\Qbf)_{j}$ are identified with zeroes, this is equal to 
    \[
    Q_j^{(j)}\bigl(Q_{j+1}^{(j+1)}(\cdots Q_k^{(k)}(\mathbf{w}_0)\cdots)\bigr).
    \]
\end{cor}

\begin{proof}
    The pairing procedure satisfies (i) by \cref{lem:straight MLQ}, (ii) by \cref{thm:pi invariant of R} below, and (iii) by \cref{prop:projection map equivalent} below. Thus invoking \cref{prop:projection maps equivalent} yields the result.
\end{proof}

\section{Multiline queue projection maps via the combinatorial R matrix}\label{sec:R}

\subsection{Fermionic projection map via the combinatorial R matrix}
\label{sec:fermionic_combR}

Recall that we have identified a (fermionic) multiline queue $\Qbf = (Q_1,\ldots,Q_k)$ of type $\alpha$ with an element $b_1\otimes \cdots \otimes b_k$ of the crystal $\calB^{\alpha}$.

In~\cite{KMO15}, a projection map is given from $\calB^{\alpha}$ to a state of the TASEP of type $\alpha$, when $\alpha$ is a partition, as a composition of combinatorial R matrices. We shall write $\combR_{[a,b]} := \combR_a\circ\cdots\circ \combR_b$ to represent a sequence of operators (operators are always applied from right to left).

\begin{defn}[Fermionic KMO projection map~\cite{KMO15}]\label{def:KMO projection fermionic}
    For $\bb = b_1\otimes \cdots \otimes b_k \in \calB^{\alpha_1} \otimes \cdots \otimes \calB^{\alpha_k}$,  define 
    \begin{equation}\label{eq:pi_ferm}
    \pi_i^{(j)}(\bb) = \Big(\combR_{[j,i-1]}(\bb)\Big)_j
    \end{equation}
    where for an element $X \in \calB^{\ell_1}\otimes\cdots\otimes \calB^{\ell_k}$, we write $(X)_j$ for the $j$-th component.
    Define $\pi^{(j)} = \sum_{\ell=j}^k \pi_{\ell}^{(j)}$.
    When $j = 1$, we simply write $\pi_i = \pi^{(1)}_i$.

    Then the projection map $\pi \colon \calB^{\alpha_1} \otimes \cdots \otimes \calB^{\alpha_k} \rightarrow S(\alpha)$ is given by
    \[
    \pi(\bb) := \pi^{(1)}(\bb) = \pi_1(\bb)+\cdots+\pi_k(\bb) = \sum_{i=1}^k\Big( \combR_{[1,i-1]}\bb\Big)_1.
    \]
\end{defn}

We can visualize the projection operators as
\begin{equation}
\label{eq:projection_picture}
\begin{tikzpicture}[scale=0.7,baseline=-5]
\node (sn) at (6,0) {$b_n$};
\node (snm) at (5,0) {$b_{n-1}$};
\node (sd) at (3.5,0) {$\cdots$};
\node (s2) at (2,0) {$b_2$};
\node (s1) at (1,0) {$b_1$};
\node (t1) at (1,-2) {$b'_1$};
\node (t2) at (2,-2) {$b'_2$};
\node (td) at (3.5,-2) {$\cdots$};
\node (tnm) at (5,-2) {$b'_{n-1}$};
\node[anchor=east] (tn) at (0,-1) {$b'_n = \pi_n$};
\draw[-] (sn) -- (6,-1) -- (tn);
\draw[-] (s1) -- (t1);
\draw[-] (s2) -- (t2);
\draw[-] (snm) -- (tnm);
\end{tikzpicture}
\end{equation}
where each crossing is a combinatorial R matrix.

\begin{remark}
This projection $\pi$ and~\eqref{eq:projection_picture} matches the description given just after~\cite[Ex.~4.3]{KMO16}.
To be more precise, the projection map in~\cite{KMO15} is actually given by the longer expression 
\begin{equation}
\label{eq:KMO_CTM}
\pi_i(\bb) = \Big(\combR_{[1,i-1]}\circ \combR_{[1,i-2]}\circ\cdots\circ \combR_{[1,2]}\circ \combR_{[1,1]}(\bb)\Big)_1.
\end{equation}
However, since the combinatorial R matrix satisfies the braid relations (Yang--Baxter and commutation relations), the above is equivalent to
\begin{equation}
\label{eq:KMO_CTM_afine}
\pi_i(\bb) = \Big(\combR_{[i-1,i-1]} \circ \combR_{[i-2,i-1]} \circ\cdots\circ \combR_{[2,i-1]} \circ \combR_{[1,i-1]}(\bb)\Big)_1.
\end{equation}
Because we only need the first tensor, the operators following (to the left of) the $\combR_{[1,i-1]}$ component are irrelevant. Thus the above is equal to our definition~\eqref{eq:pi_ferm}.

Additionally, the composition of combinatorial R matrices in~\eqref{eq:KMO_CTM} is known as a \emph{corner transfer matrix} due to its pictorial description.
Thus, the Yang--Baxter equation gives the equivalence of the two corner transfer matrices:
\[
\begin{tikzpicture}[baseline=-55,scale=.7]
\foreach \i in {1,2,3,4} {
    \node (s\i) at (0,-\i) {$b_{\i}$};
    \node (t\i) at (-\i,-5) {$b'_{\i}$};
    \draw[-] (s\i) -- (-\i,-\i) node[anchor=south east] {$\pi_{\i}$} -- (t\i);
}
\end{tikzpicture}
\qquad \longleftrightarrow \qquad
\begin{tikzpicture}[baseline=-40,scale=.7]
\foreach \i in {1,2,3,4} {
    \node (s\i) at (\i,0) {$b_{\i}$};
    \node (t\i) at (0,\i-5) {$b'_{\i}$};
    \draw[-] (s\i) -- (\i,\i-5) -- (t\i);
}
\node[anchor=east] at (t4) {$\pi_4 = \;$};
\end{tikzpicture}
\]
with the left hand side is~\eqref{eq:KMO_CTM} and the right is~\eqref{eq:KMO_CTM_afine}.
Indeed, the entire projection map is encapsulated by the corner transfer matrix~\eqref{eq:KMO_CTM}, where the elements $\pi_i$ appear on the corners.

A full example of this corner transfer matrix is given below in \cref{ex:KMO projection}.
\end{remark}

\begin{lemma}[{\cite{KMO15}}]\label{lem:KMO projection equivalent to FM}
    If $\bb$ corresponds to a straight multiline queue, then $\pi(\bb)$ is equal to the output produced by the FM algorithm.
\end{lemma}

In fact, one can extend the definition of $\pi(\bb)$ in \eqref{eq:pi_ferm} to $\bb\in \calB^{\alpha_1} \otimes \cdots \otimes \calB^{\alpha_k}$, where $\alpha$ is any composition. In doing so, we recover the result~\cite[Thm.~3.1]{AGS20} (see also \cref{ex:KMO projection} below).

\begin{theorem}\label{thm:pi invariant of R}
Let $\alpha=(\alpha_1,\ldots,\alpha_k)$ be a composition. Let $\bb\in \calB^{\alpha}$, and let $\Qbf$ be the corresponding multiline queue of type $\alpha$. For any $1\leq i<k$,
\[
    \pi\bigl(\combR_i(\bb)\bigr) = \pi(\bb),
    \qquad\text{ or equivalently }\qquad
    \Phi\bigl(\sigma_i(\Qbf)\bigr) = \Phi(\Qbf).
\]
\end{theorem}

\begin{proof}
    In fact, we will show that $\pi_j(\combR_i(\bb)) = \pi_{s_i(j)}(\bb)$, which decomposes as
    \begin{equation}
    \pi_j(\combR_i(\bb)) = \begin{cases}
        \pi_j(\bb)& \mbox{if}\ j\neq i,i+1,\\
        \pi_{i+1}(\bb)& \mbox{if}\ j=i,\\
        \pi_{i}(\bb)& \mbox{if}\ j=i+1.
    \end{cases}
    \end{equation}
If $j > i+1$, we have
\[
\pi_j(\combR_i\bb) = \Big(\combR_{[1,j-1]} \circ \combR_i(\bb)\Big)_1 = \Big(\combR_i \circ \combR_{[1,j-1)}(\bb)\Big)_1 = \pi_j(\bb).
\]
If $j<i$, we have
\[
\pi_j(\combR_i\bb) = \Big(\combR_{[1,j-1]} \circ \combR_i(\bb)\Big)_1 = \Big(\combR_{i+1} \circ \combR_{[1,j-1]}(\bb)\Big)_1 = \pi_j(\bb).
\]
If $j=i$, we have
\[
\pi_i(\combR_i\bb) = \Big(\combR_{[1,i-1]} \circ \combR_i(\bb)\Big)_1 = \Big(\combR_{[1,i]}(\bb)\Big)_1 = \pi_{i+1}(\bb).
\]
If $j=i+1$, we have
\[
\pi_{i+1}(\combR_i\bb) = \Big(\combR_{[1,i]} \circ \combR_i(\bb)\Big)_1 = \Big(\combR_{[1,i-1]}(\bb)\Big)_1 = \pi_{i}(\bb).
\]
Thus
\[
\pi(\combR_i(\bb)) = \sum_{r=1}^k \pi_r(\combR_i(\bb)) = \sum_{r=1}^k \pi_r(\bb) = \pi(\bb)
\]
as desired.
\end{proof}

We remark that a pictorial proof of \cref{thm:pi invariant of R} is essentially the corner transfer matrix diagram at the end of~\cite[Sec.~7]{AGS20}.

\begin{example}\label{ex:KMO projection}
We show the output of the corner transfer matrix on the multiline queues $\Qbf$ and $\Qbf'$ from \cref{ex:MLQ labelling procedure}, corresponding to the tensors $\bb\in \calB^3\otimes \calB^4\otimes \calB^1\otimes \calB^4$ and $\bb'\in \calB^4\otimes \calB^4\otimes \calB^3\otimes \calB^1$, respectively.
\begin{gather*}
\ytableausetup{boxsize=1em,aligntableaux=top}
\bb = \ytableaushort{1,2,4} \otimes \ytableaushort{1,3,5,6} \otimes \ytableaushort{2} \otimes \ytableaushort{1,2,3,5}\,,
\qquad
\bb'= \ytableaushort{1,2,4,5} \otimes \ytableaushort{1,2,3,6} \otimes \ytableaushort{1,3,5} \otimes \ytableaushort{2}\,,
\allowdisplaybreaks\\
\ytableausetup{aligntableaux=center}
\combR_1(\bb) = \ytableaushort{{\bullet}{\bullet}{\bullet}{}{\bullet}{},{}{\bullet}{}{}{}{},{\bullet}{}{\bullet}{}{}{\bullet},{\bullet}{\bullet}{}{\bullet}{\bullet}{}}\,,
\qquad
\combR_{[1,2]}(\bb) = \ytableaushort{{\bullet}{\bullet}{\bullet}{}{\bullet}{},{\bullet}{\bullet}{}{}{\bullet}{\bullet},{\bullet}{\bullet}{\bullet}{}{}{},{}{}{}{\bullet}{}{}}\,,
\qquad
\combR_{[1,3]}(\bb) = \ytableaushort{{}{\bullet}{}{}{}{},{\bullet}{\bullet}{\bullet}{}{\bullet}{},{\bullet}{}{\bullet}{}{}{\bullet},{\bullet}{\bullet}{}{\bullet}{\bullet}{}}\,.
\end{gather*}
Thus
$\pi_1(\bb)=(\bb)_1=(1,1,0,1,0,0)$, $\pi_2(\bb)=(\combR_1(\bb))_1=(1,1,0,1,1,0)$, $\pi_3(\bb)=(\combR_{[1,2]}(\bb))_1=(0,0,0,1,0,0)$, $\pi_4(\bb)=(\combR_{[1,3]}(\bb))_1=(1,1,0,1,1,0)$, and $\pi(\bb)=(3,3,0,4,2,0)$.

Similarly,
\[
\combR_1(\bb')=\bb',
\qquad
\combR_{[1,2]}(\bb') = \ytableaushort{{}{\bullet}{}{}{}{},{\bullet}{\bullet}{\bullet}{}{\bullet}{},{\bullet}{}{\bullet}{}{\bullet}{\bullet},{\bullet}{\bullet}{}{\bullet}{}{}}\,,
\qquad
\combR_{[1,3]}(\bb') = \ytableaushort{{\bullet}{\bullet}{}{}{\bullet}{},{\bullet}{\bullet}{\bullet}{}{}{\bullet},{\bullet}{\bullet}{\bullet}{}{\bullet}{},{}{}{}{\bullet}{}{}}\,,
\]
Thus $\pi_1(\bb')=\pi_2(\bb')=(1,1,0,1,1,0)$, $\pi_3(\bb')=(1,1,0,1,0,0)$, and $\pi_4(\bb')=(0,0,0,1,0,0)$. Indeed, $\pi_j(\bb')=\pi_{\tau(j)}(\bb)$, where $\tau=s_2s_3s_1\in S_4$ is the permutation that sends $B$ to $B'$. In particular, $\pi(\bb) = \pi(\bb') = \Phi(\Qbf) = \Phi(\Qbf')$.
\end{example}

\begin{lemma}\label{lem:decompose B}
    Let $\uu\in\fermwords$ be a word with $\min \uu>1$ 
    and $Q\subseteq[n]$ a queue corresponding to the column $\Qbf=\iota(Q)$. Let $\uu=u_1+\cdots+u_m\in\fermwords$, where $(0,\ldots,0)<u_m\leq u_{m-1}\leq \cdots\leq u_2= u_1$. Then  
    \begin{equation}\label{eq:B}
    Q^{(1)}(\uu)= \Qbf + \sum_{i=2}^{m} \Big(\combR(\Qbf\otimes u_i)\Big)_1.
    \end{equation}
\end{lemma}

\begin{proof}
    We prove the equality by comparing the sets $v^{(i)}$ from \cref{def:mlq pairing} to the terms $\big(\combR(\Qbf\otimes u_i)\big)_1$ in \cref{rem:direct labelling}. 
    
    First, we observe that $\iota(u^{(i)})=u_i$, since  $\iota(u^{(m)})\leq \cdots\leq \iota(u^{(1)})$ by definition, and the decomposition $\uu=u_1+\cdots+u_m$ with $0 < u_m\leq \cdots \leq u_1$  is unique. 

    Now, we examine the $v^{(i)}$'s. First suppose that $\abs{\uu} > \abs{Q}$, so there is a nonempty set of labels pairing during the collapsing phase; let $L>1$ be the largest such label. From \cref{rem:direct labelling}, we observe that $v^{(i)}$ is the set of indices of particles paired below in $\Par(Q,u^{(i)})$ plus particles unpaired above in $\Par(Q,u^{(i+1)})$, which is equivalent to writing
    \begin{equation}\label{eq:vi}
    \iota(v^{(i)})=\begin{cases}
        \big(\combR(\Qbf\otimes u_i)\big)_1& \text{if } i>L,\\
        \Qbf& \text{if } i=L,\\
        \big(\combR(\Qbf\otimes u_{i+1})\big)_1& \text{if } i<L.
    \end{cases}
    \end{equation}
    Similarly, if $\abs{\uu} \leq \abs{Q}$, all labels in $\uu$ pair during the particle phase so that $\iota(v^{(i)}) = \big(\combR(\Qbf\otimes u_i)\big)_1$ for $i\geq 2$, and the algorithm is completed by setting $\iota(v^{(1)}) = \Qbf$.
    
    Therefore $Q^{(1)}(\uu) = \iota(v^{(1)})+\cdots+\iota(v^{(m)})$ is indeed equal to the expression on the right hand side of~\eqref{eq:B}.
\end{proof}    
    
\begin{example}
To see \cref{lem:decompose B} in action, we examine \cref{ex:MLQ direct labelling} with $\uu=3\,2\,0\,4\,2\,5\,4\,3\,3\,0\,3$, $Q=\{2,4,5,8,10\}$, and $\Qbf = \iota(Q) = (0,1,0,1,1,0,0,1,0,1,0)$, so that $L=3$. This gives the following.

    \begin{center}
\begin{tabular}{c|c|r}
    i&$u_i$&\multicolumn{1}{c}{$\iota(v^{(i)})$}\\\hline
    5&    $(0,0,0,0,0,1,0,0,0,0,0)$&$(\combR(\Qbf\otimes u_5))_1=(0,0,0,0,0,0,0,1,0,0,0)$\\
    4&$(0,0,0,1,0,1,1,0,0,0,0)$
       & $(\combR(\Qbf\otimes u_4))_1=(0,0,0,1,0,0,0,1,0,1,0)$\\
    3&$(1,0,0,1,0,1,1,1,1,0,1)$
       & $\Qbf=(0,1,0,1,1,0,0,1,0,1,0)$\\
    2&$(1,1,0,1,1,1,1,1,1,0,1)$
       & $(\combR(\Qbf\otimes u_3))_1=(0,1,0,1,1,1,1,1,0,1,0)$\\
        1&$(1,1,0,1,1,1,1,1,1,0,1)$
       & $(\combR(\Qbf\otimes u_2))_1=(1,1,0,1,1,1,1,1,0,1,1)$
    \end{tabular}
    \end{center}
    Indeed, $\iota(v^{(i)})$ is given by \eqref{eq:vi}, and $Q^{(1)}(\uu)=\iota(v^{(5)})+\cdots+\iota(v^{(1)})=(1,3,0,4,3,2,2,5,0,4,1)$.
\end{example}

We can explicitly obtain the output of the KMO projection map using the following expression.

\begin{prop}\label{prop:projection map equivalent}
Let $\Qbf=(Q_1,\ldots,Q_k)$ be a multiline queue corresponding to the tensor $\s=\Qbf_1\otimes\cdots\otimes \Qbf_k$ (with $\Qbf_j=\iota(Q_j)$). Then
\begin{equation}\label{eq:Q and pi}
Q_j^{(1)}\bigl(Q_{j+1}^{(2)}(\cdots Q_k^{(k-j+1)}(\mathbf{w}_0)\cdots)\bigr) = \pi^{(j)}(\bb).
\end{equation}
Moreover, 
\begin{equation}\label{eq:Q and pi incremented}
Q_j^{(j)}\bigl(Q_{j+1}^{(j+1)}(\cdots Q_k^{(k)}(\mathbf{w}_0)\cdots)\bigr) = \pi^{(j)}(\bb)^{(+j-1)}.
\end{equation}
\end{prop}

\begin{proof}
Our proof is by induction in $j$, with the base case being $j=k$, where indeed $Q^{(1)}_k(\mathbf{w}_0)=\iota(Q_k)=\pi^{(k)}_k(\bb)=\pi^{(k)}(\bb)$. 

Next, we claim that for any $Q$ and $v\in\fermwords$ with $\ell<\min v$, 
\begin{equation}\label{eq:increment}
Q^{(\ell)}(v)^{(+1)}=Q^{(\ell+1)}(v^{(+1)}).
\end{equation}
Indeed, we have that if $v=v_1+v_1+\cdots+v_m$ then $v^{(+1)}=v_1+v_1+v_2+\cdots+v_m$, so the steps of the pairing procedures for $Q^{(\ell)}(v)$ and $Q^{(\ell+1)}(v^{(+1)})$ coincide with the exception of one additional step in the latter which effectively increments all of the labels in $Q^{(\ell)}(v)$, giving $Q^{(\ell)}(v)^{(+1)}$.

Suppose $j<k$ is such that $Q_{j+1}^{(1)}\bigl(Q_{j+2}^{(2)}(\cdots Q_k^{(k-j)}(\mathbf{w}_0)\cdots)\bigr) = u$, where $u=\pi^{(j+1)}(\bb)$. By \eqref{eq:increment}, 
\[
Q_{j}^{(1)}\bigl(Q_{j+1}^{(2)}(\cdots Q_k^{(k-j+1)}(\mathbf{w}_0)\cdots)\bigr) = Q_j^{(1)}(u^{(+1)}).
\]
We will show that $Q_j^{(1)}(u^{(+1)})=\pi^{(j)}(\bb)$.

Let $u=u_1+\cdots+u_{k-j}$ be the decomposition $u$ into nested indicator vectors  where 
$(0,\ldots,0)<u_{k-j}\leq \cdots\leq u_1$. It is then the case that 
\[\Big\{\pi^{(j+1)}_{j+1}(\bb),\pi^{(j+1)}_{j+2}(\bb),\ldots,\pi^{(j+1)}_{k}(\bb)\Big\}=\{u_1,u_2,\ldots,u_{k-j}\}\] 
since the $\pi^{(j+1)}_{\ell}(\bb)$'s are also a set of nested indicator vectors summing to $u$, and such a decomposition is unique. Writing $u^{(+1)}=u_1'+\cdots+u'_{k-j+1}$, we then have $u'_j=u_{j-1}$ for $j\geq 2$ and $u'_1=u_1$. Then we have the following expansion:
\begin{subequations}
\begin{align}
    \pi^{(j)}(\bb) =
\sum_{\ell=j}^{k}\pi^{(j)}_\ell(\bb)
= (\bb)_j + \sum_{\ell=j}^{k-1}\Big(\combR_{[j,\ell]}(\bb) \Big)_j
&= \Qbf_j + \sum_{\ell=j}^{k-1}\Big(\combR_j(\combR_{[j+1,\ell]}(\bb)) \Big)_j\nonumber\\
&= \Qbf_j + \sum_{i=j+1}^{k}\Big(\combR(\Qbf_j\otimes \pi^{(j+1)}_{i}(\bb)) \Big)_1 \label{eq:f2}\\
&= \Qbf_j + \sum_{i=1}^{k-j} \Big(\combR(\Qbf_j\otimes u_i) \Big)_1 \label{eq:f1}\\
&= \Qbf_j+\sum_{i=2}^{k-j+1} \Big(\combR(\Qbf_j\otimes u'_i) \Big)_1 \label{eq:f3}\\
& = Q_j^{(1)}(u^{(+1)}). \label{eq:f4}
\end{align}
\end{subequations}
Above, the equality~\eqref{eq:f2} is due to the fact that $\Big(\combR_j(b_1 \otimes \cdots \otimes b_k)\Big)_j = \Big(\combR(b_j \otimes b_{j+1})\Big)_1$, the equality~\eqref{eq:f3} is by our induction assumption, the equality~\eqref{eq:f1} is by the definition of $u'_j$, and the equality~\eqref{eq:f4} is by \cref{lem:decompose B}. Thus we have \eqref{eq:Q and pi} by induction, and we have~\eqref{eq:Q and pi incremented} by iteratively applying~\eqref{eq:increment}.
\end{proof}

In particular, \cref{thm:equivalent_proj_AAMP} implies that the output of the projection map of \cref{def:mlq pairing} coincides with the projection map $\pi$ of \cref{def:KMO projection fermionic}.

\begin{example}
    \label{ex:partial_CTM}
    Consider the multiline queue $\Qbf=(\{1, 2, 4\}, \{1, 3, 5, 6\}, \{2\}, \{1, 2, 3, 5\})$.
    We illustrate \cref{prop:projection map equivalent} for $\Qbf$ in \cref{ex:MLQ labelling procedure}:
    \begin{center}
        \begin{tabular}{c|c|c|c}
        $j$&$Q_j^{(1)}(\cdots Q_k^{(k-j+1)}(\mathbf{w}_0)\cdots)$&$Q_j^{(j)}(\cdots Q_k^{(k)}(\mathbf{w}_0)\cdots)$&$\pi^{(j)}(\bb)=\pi^{(j)}_j(\bb)+\cdots+\pi^{(j)}_k(\bb)$\\\hline
        4&$(1,1,1,0,1,0)$&$(4,4,4,0,4,0)$&$(1,1,1,0,1,0)$   \\
        3&$(1,2,1,0,1,0)$&$(3,4,3,0,3,0)$&$(1,2,1,0,1,0)$ \\ 
        2&$(2,0,3,0,2,2)$&$(3,0,4,0,3,3)$&$(2,0,3,0,2,2)$\\
        1&$(3,3,0,4,2,0)$&$(3,3,0,4,2,0)$&$(3,3,0,4,2,0)$
        \end{tabular}
    \end{center}
\end{example}

By setting $j=1$ in \cref{prop:projection map equivalent}, we get the following.

\begin{cor}\label{thm:projection is KMO}
    Let $\Qbf = (Q_1,\ldots,Q_k)$ be a multiline queue corresponding to $\bb = Q_1 \otimes \cdots \otimes Q_k$. Then 
    \[
        \Phi(\Qbf) = \pi(\bb).
    \]
\end{cor}

As a corollary of \cref{thm:projection is KMO} and \cref{thm:pi invariant of R}, we have that the projection map is invariant under the action of $\sigma_i$.

\begin{cor}\label{cor:projection sigma invariant}
  Let $M$ be a multiline queue on $k$ rows, and let $1\leq i<k$. Then $\Phi(M)=\Phi(\sigma_i(M))$. 
\end{cor}

Combined with \cref{thm:equivalent_proj_AAMP}, \cref{cor:projection sigma invariant} recovers~\cite[Thm.~3.1]{AGS20}.

\begin{remark}\label{rem:full labelling}
In fact, \eqref{eq:Q and pi incremented} describes the full labelling procedure of \cref{def:mlq pairing} on a twisted multiline queue $\Qbf = (Q_1,\ldots,Q_k)$. For $j=k,k-1,\ldots,1$, the labelled row $j$ produced as an intermediate step by the procedure corresponds to the output given by the sequence of operators $Q_j^{(j)}(\cdots Q_k^{(k)}(\mathbf{w}_0)\cdots)$, which is equivalent to $\Phi(\Qbf_{[j,k]})^{(+j-1)}$ where $\Qbf_{[j,k]} = (Q_j,Q_{j+1},\ldots,Q_k)$ is the partial multiline queue consisting of rows $j$ and above. Note that each entry in $\Phi(\Qbf_{[j,k]})=Q_j^{(1)}(\cdots Q_k^{(k-j+1)}(\mathbf{w}_0)\cdots)$ needs to be incremented by $j-1$ to match the output $Q_j^{(j)}(\cdots Q_k^{(k)}(\mathbf{w}_0)\cdots)$. Up to this incrementation, for a given row $j$, this intermediate step can be interpreted as a (partial) corner transfer matrix that disregards any input below row $j$ and takes its readings $\pi^{(j)}_\ell$ for $j\leq \ell<k$ at row $j$.
In particular, we can see this procedure in the rightmost column in~\cref{ex:partial_CTM}.
\end{remark}

\subsection{Bosonic projection map via the combinatorial R matrix}

The projection map for fermionic multiline queues has an analogous construction for bosonic multiline queues, which was introduced in~\cite{KMO16}. Define the set of multisets of size $\ell$ in $n$ types: 
\[
B_{\ell} = \{ (v_1, \ldots, v_n)\in \NN^n \mid v_1 + \cdots + v_n = \ell \} \subset \mathcal{M}(n).
\]
With $B_{\ell}$ considered as a $U_q(\asl_n)$-crystal, its tensor product is the crystal $\wcalB(\alpha) \vcentcolon= \wcalB^{\alpha_1}\otimes \cdots \otimes \wcalB^{\alpha_k}$.
Now we can write the bosonic multiline queue $D=(D_1,\ldots,D_k)$ of type $\alpha$ as an element $\bb_1 \otimes \cdots \otimes \bb_k\in \wcalB(\alpha)$, where $\bb_i=\iota(D_i)
\in \wcalB^{\alpha_i}$ is the indicator vector corresponding to $D_i$.

\begin{remark}
Our definition of $\wcalB(\alpha)$ departs from that in~\cite{KMO16} since we let the parts of $\alpha$ correspond directly to the cardinalities of the tensors, whereas the $\alpha$ in $\wcalB(\alpha)$ of~\cite{KMO16} gives the multiplicity of the parts of $\lambda'$, where $\lambda=\alpha^+$, and the crystals discussed therein require the cardinalities of the tensors to be weakly decreasing. Our definition is necessary to define the more general set of objects, which are the bosonic multiline queues.
\end{remark}

\begin{defn}[{Bosonic KMO projection map~\cite{KMO16}}]\label{def:KMO projection bosonic}
    For $\bb = b_1 \otimes \cdots \otimes b_k \in \wcalB^{\alpha_1} \otimes \cdots \otimes \wcalB^{\alpha_k}$,  define 
    \begin{equation}\label{eq:pi_boson}
    \pi_i^{(j)}(\bb) = \Big(\combR_{[j,i-1]}(\bb)\Big)_j
    \end{equation}
    and $\pi^{(j)} := \sum_{\ell=j}^k \pi_{\ell}^{(j)}$.
    For brevity, we write $\pi_i = \pi_i^{(1)}$.

    Then the projection map $\pi \colon \wcalB^{\alpha_1} \otimes \cdots \otimes \wcalB^{\alpha_k}\rightarrow \bosonwords$ is given by
    \[
    \pi(\bb) := \pi^{(1)}(\bb) = \pi_1(\bb) + \cdots + \pi_k(\bb) = \sum_{i=1}^k \Big( \combR_{[1,i-1]}\bb\Big)_1.
    \]
\end{defn}

\begin{remark}\label{rem:bosonic po invariant}
    As in the fermionic case, the bosonic projection map defined in~\cite{KMO16} was restricted to straight bosonic multiline queue. However, the definition naturally extends to twisted diagrams using the fact that $\pi(\bb)=\pi(\combR_i \bb)$ for all $1 \leq i < k$ by~\cref{thm:pi invariant of R}, which relies only on the fact that the combinatorial R matrix satisfies the Yang--Baxter equation.
    Moreover, it was also only given as a corner transfer matrix, but we can recover a labeling procedure from~\cref{rem:full labelling}.
\end{remark}

We get the bosonic analog of~\cref{thm:projection is KMO} using the same reasoning as in the fermionic case.

\begin{theorem}\label{thm:bosonic projection is KMO}
    Let $\Dbf = (D_1,\ldots,D_k)$ be a bosonic multiline queue corresponding to $\bb \in \wcalB^{\alpha}$. Then 
    \[
        \widetilde{\Phi}(\Dbf) = \pi(\bb).
    \]
\end{theorem}

\begin{proof}
The proof is identical to its fermionic analog, modulo some notation. First, we will show that for a bosonic queue $D\in\mathcal{M}(n)$ corresponding to the tensor $\bb \in \wcalB^{\alpha}$ (the column vector representing $\iota(D)$) and a bosonic word $\uu\in\bosonwords$, 
\begin{equation}\label{eq:BB}
    D^{(1)}(\uu)= \bb \oplus \bigoplus_{i=2}^{m} \Big(\combR(\bb \otimes U_i)\Big)_1,
\end{equation}
where $U_k\leq \cdots\leq U_1\in\mathcal{M}(1)^n$ with $U_j=\nu_j(\uu)$ is the decomposition of $\uu$ into nested bosonic indicator vectors, so that $U=U_1\oplus\cdots\oplus U_k$. Indeed, the same argument as for \cref{lem:decompose B} holds due to \cref{rem:direct labelling MLD}. 

Next, let $\Dbf=(D_1,\ldots,D_k)$ be a bosonic multiline queue with $D_i\in\mathcal{M}(n)$, and let $\bb = b_1\otimes\cdots\otimes b_k$ be the corresponding tensor (recall that $b_j$ is the column vector corresponding to $\iota(D_j)$). We use induction on $j$ to show that
\[
D_j^{(1)}\bigl(D_{j+1}^{(2)}(\cdots D_k^{(k-j+1)}(\bwt_0)\cdots)\bigr) = \pi^{(j)}(\bb),
\]
where $\bwt_0=(\emptyset,\ldots,\emptyset)\in\bosonwords$. The full result follows by taking $j=1$.
\end{proof}

We immediately obtain the bosonic analogue of \cref{cor:projection sigma invariant}.

\begin{cor}\label{cor:bosonic Phi sigma invariant}
Let $\Dbf$ be a bosonic multiline queue and let $1\leq i<L$. Then $\widetilde{\Phi}\bigl(\sigma_i(\Dbf)\bigr) = \widetilde{\Phi}(\Dbf)$.
\end{cor}

\begin{proof}
Following the explanation in \cref{rem:bosonic po invariant}, we have
\[
\widetilde{\Phi}\bigl(\sigma_i(\Dbf)\bigr) = \pi(\combR_i\bb) = \pi(\bb) = \widetilde{\Phi}(\Dbf)
\]
as claimed.
\end{proof}

\section{Markov chain on bosonic twisted multiline queues that projects to the 0-TAZRP}\label{sec:bMLQ Markov}

In this section, we define a bosonic analogue of the Ferarri-Martin ringing paths to obtain a Markov chains on bosonic twisted multiline queues, which project to the 0-TAZRP.

For a fixed $(\lambda,n)$, define the functions 
\[
\Ff,\Rf\colon \bMLQ(\lambda,n)\times \ZZ_n\rightarrow \bMLQ(\lambda,n)\times\ZZ_n,
\]
where $\Ff(M,i)$ gives a \dfn{forward transition} and $\Rf(M,i)$ gives a \dfn{reverse transition} corresponding to the time reversal process. Define $\Ff_i(M)$ and $\Rf_i(M)$ to be the multiline queues obtained from these transitions.

\begin{defn}
\label{def:tazrp ringing}
    Let $\lambda$ be a partition with $L:=\lambda_1$. For a bosonic multiline queue $\Dbf=(D_1,\ldots,D_L) \in \bMLQ(\lambda,n)$ and a site number $i\in\ZZ_n$, a \dfn{ringing path} in $\Dbf$ at site $i$ is a sequence of sites $\fa_1,\ldots,\fa_{L+1}$, where $\fa_1=i$ and $\fa_j$ for $2\leq j\leq L+1$ is defined recursively by:
    \begin{equation}\label{eq:fa}
    \fa_j=\begin{cases}
        \fa_{j-1} & \text{if } \fa_{j-1}\not\in D_{j-1},\\
        \fa_{j-1}+1 & \text{otherwise}.
    \end{cases}
    \end{equation}
    Then $\Ff(\Dbf,i)=(\mathbf{D'},\fa_{L+1}-1)$, where $\mathbf{D'}=(D'_1,\ldots,D'_L)$ is given by
    \[
    D'_j=\begin{cases}
        D_j& \text{if } \fa_j\notin D_j, \\
        D_j\cup \{\fa_j+1\}\setminus \{\fa_j\} & \text{if } \fa_j \in D_j.
    \end{cases}
    \]
    In other words, for $1\leq j\leq L$, if site $\fa_j$ of row $j$ in $\mathcal{D}$ is nonempty, then one particle hops from site $\fa_j$ to site $\fa_{j}+1$. The \emph{rate} of the transition from $\Dbf$ to $\mathbf{D'}$ is denoted $\rate(\Dbf, \mathbf{D'})$, and is equal to 1 if column $i$ is empty, and $x_i^{-1}$ otherwise.
\end{defn}

The reverse ringing path is defined similarly.
\begin{defn}
\label{def:tazrpreverse ringing}
    For a bosonic multiline queue $\Dbf=(D_1,\ldots,D_L) \in \bMLQ(\lambda,n)$ and a site number $i\in\ZZ_n$, a \dfn{reverse ringing path} in $\Dbf$ at site $i$ is a sequence of sites $\fb:=(\fb_L,\ldots,\fb_{0})$, where $\fb_L=i$ and $\fb_j$ for $0\leq j\leq L-1$ is defined recursively by
    \begin{equation}\label{eq:fb}
    \fb_j=\begin{cases}
        \fb_{j+1} & \text{if } \fb_{j+1}+1\not\in D_{j+1},\\
        \fb_{j+1}-1 & \text{otherwise}.
    \end{cases}
    \end{equation}
    Then $\Rf(\Dbf,i)=(\mathbf{D''},\fb_{0}+1)$, where $\mathbf{D''}=(D''_1,\ldots,D''_L)$ is given by
    \[
    D''_j=\begin{cases}
        D_j& \text{if } \fb_j+1\notin D_j,\\
        D_j\cup \{\fb_j\}\setminus \{\fb_j+1\}& \text{if } \fb_j+1\in D_j.
    \end{cases}
    \] 
    In other words, for $1\leq j\leq L$, if site $\fb_j+1$ of row $j$ of $\Dbf$ is nonempty, a particle hops from site $\fb_j+1$ to site $\fb_{j}$. The \emph{rate} of the reverse transition from $\Dbf$ to $\mathbf{D''}$ (at site $i$) is denoted $\rate_R(\Dbf, \mathbf{D''})$, and is equal to 1 if column $i$ is empty, and $x_{i+1}^{-1}$ otherwise.
\end{defn}

\begin{example}
    Consider the bosonic multiline queues below, with integers representing the numbers of particles at each site:
    \[
    \ytableausetup{boxsize=1em}
    \Dbf = \ytableaushort{{}11{},3{}{}{},13{}{},21{}2}, \quad
    \Dbf_1 = \ytableaushort{{}1{}1,3{}{}{},121{},12{}2}, \quad
    \Dbf_2 = \ytableaushort{{}1{}1,3{}{}{},13{}{},2{}12}, \quad
    \Dbf_3 = \ytableaushort{{}1{}1,3{}{}{},13{}{},21{}2}, \quad
    \Dbf_4 = \ytableaushort{{}{}2{},3{}{}{},{}4{}{},31{}1}.
    \]
    Then the ringing path at site 1 is $\fa_1=1, \fa_2=2, \fa_3=\fa_4=3, \fa_5=4$, so that $\Ff(\Dbf,1)=(\Dbf_1,3)$. The ringing path at site 2 is $\fa_1=2,\fa_2=\fa_3=\fa_4=3,\fa_5=4$, so that $\Ff(\Dbf,2)=(\Dbf_2,3)$. The ringing path at site 3 is $\fa_1=\fa_2=\fa_3=\fa_4=3, \fa_5=4$, so that $\Ff(\Dbf,3)=(\Dbf_3,3)$. Finally, the ringing path at site 4 is $\fa_1=4, \fa_2=1, \fa_3=\fa_4=2, \fa_5=3$, so that $\Ff(\Dbf,4)=(\Dbf_4,2)$. We check that the reverse ringing path at site 2 of $\Dbf_4$ is $\fb_4=2, \fb_3=\fb_2=1, \fb_1=4, \fb_0=3$, confirming that $\Rf(\Ff(\Dbf,4))=\Rf(\Dbf_4,2)=(\Dbf,4)$. For the rates, we have $\rate(\Dbf, \Dbf_1)=x_1^{-1}$, $\rate(\Dbf, \Dbf_2)=x_2^{-1}$, $\rate(\Dbf, \Dbf_3)=x_3^{-1}$, and $\rate(\Dbf, \Dbf_4)=x_4^{-1}$.
\end{example}

\begin{remark}
    The Markov chain presented here is different from the one defined in~\cite{AMM22} for the $t$-TAZRP. The main difference lies in the fact that in our Markov chain, only the sites in the bottom row of a multiline queue can trigger a transition, so that there are at most $n$ outgoing (and incoming) transitions in total. On the other hand, in~\cite{AMM22}, every site could potentially trigger a transition. The more complicated chain is needed for the general $t$ statistic. However it does not specialize to ours at $t=0$.
\end{remark}

\begin{lemma}\label{lem:bosonic inv}
    For $\Dbf\in\bMLQ(\lambda,n)$ and $1\leq i\leq n$, we have 
    \[
    \Rf(\Ff(\Dbf,i)) = \Ff(\Rf(\Dbf,i))=(\Dbf,i).
    \]
    Moreover, 
    $x^{\Ff_i(\Dbf)}=x^\Dbf x_{j+1}x_i^{-1}$, where $\Ff(\Dbf,i)=(\Ff_i(\Dbf),j)$, and so
    \[
    \xx^{\Dbf} \rate(\Dbf,\Ff_i(\Dbf)) = \xx^{\Ff_i(\Dbf)}\rate_R(\Ff_i(\Dbf),\Dbf).
    \]
\end{lemma}

\begin{proof}
    Suppose the ringing path on $\Dbf$ triggered at site $i$ is $\fa_1,\ldots,\fa_{L+1}$, and the ringing path on $\Ff_i(\Dbf)$ triggered at site $\fa_{L+1}$ is $\fb_L,\ldots,\fb_{1}$.  We claim that $\fa_i=\fb_{i-1}+1$ for $1\leq i\leq L+1$. Suppose $\fb_{j}=\fa_{j+1}-1$ for some $1\leq j\leq L$. If $\fa_j\not\in D_j$, then $\fa_j=\fa_{j+1}=\fb_j+1$ by \eqref{eq:fa}. Since $\fb_{j-1}=\fb_j$ by \eqref{eq:fb}, we have $\fb_{j-1}=\fa_j-1$. On the other hand, if $\fa_j\in D_j$, then $\fa_j=\fa_{j+1}-1=\fb_j$ and $\fb_{j-1}=\fb_j-1$ by \eqref{eq:fa} and \eqref{eq:fb}, and so we again have $\fb_{j-1}=\fa_j-1$. By definition, $\fb_L=\fa_{L+1}-1$, and thus our claim follows by induction.

    The forward ringing path triggered by $i$ on $\Dbf$ coincides with the reverse ringing path triggered by $\fa_{L+1}-1$ on $\Ff_i(\Dbf)$ shifted to the right by one site, and the reverse transition is precisely the inverse of the forward transition on those ringing paths. Thus we have that $\Rf_{\fa_{L+1}-1}(\Ff_i(\Dbf))=\Dbf$, or equivalently, $\Rf(\Ff(\Dbf,i))=(\Dbf,i)$. A similar argument shows that $\Ff(\Rf(\Dbf,i))=(\Dbf,i)$.

    Finally, we observe that $\fa_{L+1}=\fa_1=i$ if and only if the column $i$ is empty in $\Dbf$, in which case $\Dbf=\Ff_i(\Dbf)=\Rf_j(\Dbf)$, where $j=i-1$. In this case, $\rate(\Dbf,\Ff_i(\Dbf))=\rate_R(\Ff_i(\Dbf),\Dbf)=1$. Otherwise, exactly one particle hops one site to the right for each value in the set $\{\fa_2,\ldots,\fa_L,\fa_{L+1}\}$ such that $\fa_k\neq \fa_{k-1}$, and $\fa_{L+1}$ is the site to which the topmost particle hops. Thus the total column content in column $\fa_1=i$ decreases by 1 particle, and the total column content in column $\fa_{L+1}=j+1$ increases by 1 particle. In this case, $\rate(\Dbf,\Ff_i(\Dbf))=x_i^{-1}$ and $\rate_R(\Ff_i(\Dbf),\Dbf)=x_{j+1}^{-1}$. In both cases we have $x^{\Ff_i(\Dbf)}=x^\Dbf x_{j+1}x_i^{-1}$, as desired.
\end{proof}

\begin{theorem}\label{thm:bMLQ MC}
    The Markov chain on $\bMLQ(\lambda,n)$ as defined above has an unnormalized stationary distribution $\Wt\colon \bMLQ(\lambda,n) \rightarrow \mathbb{Q}[\xx]$ given by $\Wt(\Dbf) = \xx^{\Dbf}$, the weight of $\Dbf$.
\end{theorem}

\begin{proof}
    We prove that the distribution on $\bMLQ(\alpha,n)$ given by $\Wt$ satisfies the balance condition
    \begin{equation}\label{eq:bMLQ balance}
\Wt(\Dbf)\sum_{\mathbf{D'}\in\bMLQ(\alpha,n)}\rate(\Dbf, \mathbf{D'}) = \sum_{\mathbf{D'}\in\bMLQ(\alpha,n)} \Wt(\mathbf{D'})\rate(\mathbf{D'}, \Dbf)
    \end{equation}
    for all $\Dbf\in\bMLQ(\alpha,n)$, using the time reversal process. We rewrite \eqref{eq:bMLQ balance} as
    \begin{align*}
    \Wt(\Dbf)\sum_{1\leq j\leq n}\rate(\Dbf, \Ff_j(\Dbf)) 
    &= \sum_{1\leq j\leq n} \Wt(\Rf_j(\Dbf))\rate(\Rf_j(\Dbf), \Dbf)\\
    &= \sum_{1\leq j\leq n} \Wt(\Dbf)\rate_R(\Dbf, \Rf_j(\Dbf))
    \end{align*}
    for all $\Dbf\in\bMLQ(\alpha,n)$, where the second equality is obtained from \cref{lem:bosonic inv}. Now we compare the sums on both sides:
    \[
    \sum_{1\leq j\leq n}\rate(\Dbf, \Ff_j(\Dbf))=\sum_{\substack{1\leq j\leq n\\j\in \Dbf}} x_j^{-1} =\sum_{\substack{1\leq j\leq n\\j+1\in \Dbf}} x_{j+1}^{-1}=\sum_{1\leq j\leq n}\rate_R(\Dbf, \Rf_j(\Dbf)),
    \]
    thus confirming the equality \eqref{eq:bMLQ balance} for all $\Dbf\in\bMLQ(\alpha,n)$.
\end{proof}

We now give an alternate proof of \cref{thm:tazrp} by showing our Markov chain projects (in the sense of \cref{def:projection}) to the 0-TAZRP. In fact, the result immediately generalizes to twisted multiline queues, since the ringing paths are defined for $\bMLQ$'s in general. 

\begin{theorem}\label{thm:tazrp projects}
    Let $\alpha$ be a composition that rearranges to the partition $\lambda$. Then the Markov chain on $\bMLQ(\alpha,n)$ projects to the 0-TAZRP of type $(\lambda',n)$ via the map $\Phi$. 
\end{theorem}

\begin{proof}
    Let $\Dbf\in\bMLQ(\alpha,n)$ with $\Phi(\Dbf) = \tau = \tau_1 \cdots \tau_n \in \bosonwords(\lambda',n)$. Let $1\leq j\leq n$. We will prove the following two cases:
    \begin{itemize}
    \item[i.] If $\tau_j=\emptyset$, then $\Phi(\Ff_j(\Dbf))=\tau$. 
    \item[ii.] Otherwise, $\Phi\bigl(\Ff_j(\Dbf)\bigr) = \tau'$, where $\tau'$ is obtained from $\tau$ by moving the largest element in $\tau_i$ to from site $j$ to site $j+1$.
    \end{itemize}
    Once these cases have been proven, we obtain that for any $\tau,\tau'\in\bosonwords(\lambda',n)$ such that $\rate(\tau,\tau')\neq 0$, there exists $1\leq j\leq n$ such that the sum on left hand side of \eqref{eq:projection} is over the single multiline queue $\Ff_j(\Dbf)$, and the equality $\rate(\Dbf,\Ff_j(\Dbf))=\rate(\tau,\tau')=x_j^{-1}$ is automatically achieved. By \cref{def:projection}, this implies that the stationary distribution of the 0-TAZRP is given by \eqref{eq:stationary}.

Let $\Dbf=(D_1,\ldots,D_L)$, $\mathbf{D'}=\Ff_j(\Dbf)=(D'_1,\ldots,D'_L)$, and let $\fa_1,\ldots,\fa_L$ be the ringing path on $\Dbf$ initiated at $\fa_1=j$. Our proof of both cases is by induction on the number of rows in $\Dbf$. The induction base case $L=1$ is simply noting $\Phi((D)) = \iota(D) \in \bosonwords(1^{\abs{D}}, n)$ (\textit{i.e.}, the projection is trivial). Therefore, suppose both claims hold for any multiline queue of at most $L-1$ rows. Consider the multiline queues $\mathbf{G}=(D_2,\ldots,D_L)$ and $\mathbf{G'}=\Ff_{\fa_2}(\mathbf{D'})=(D'_2,\ldots,D'_L)$, and write $\Phi(\Dbf)=D_1^{(1)}(\Phi(\mathbf{G})^{(+1)})$ and $\Phi(\mathbf{D'})=D_1'^{(1)}(\Phi(\mathbf{G'})^{(+1)})$ in the notation from \cref{sec:bosonic projection}.

If $j\not\in D_2$, by the induction hypothesis, the case in claim (i) occurs: $\Phi(\mathbf{G})=\Phi(\mathbf{G'})$, and we immediately have that $\Phi(\Dbf)=\Phi(\mathbf{D'})$, as desired. Otherwise, if $j\in D_2$, the case in claim (ii) occurs: $\Phi(\Ff_j(\mathbf{G}))$ corresponds to the largest particle at site $j$ of $\Phi(\mathbf{G})$ hopping to site $j+1$. During the pairing procedure, the only particles in $D_1$ that might be affected by the change from $\Phi(\mathbf{G})$ to $\Phi(\mathbf{G'})$ would be those at site $j$ in $D_1$; all other pairings are preserved. However, since that site is empty, we indeed have that $\Phi(\Ff_j(\Dbf))=D_1^{(1)}(\Phi(\Ff_j(\mathbf{G}))^{(+1)})=D_1(\Phi(\mathbf{G})^{(+1)})=\Phi(\Dbf)$, as desired. This completes the proof of claim (i).

For claim (ii), we will show that if $j\in D_1$ (and hence also $\tau_j\neq\emptyset$), then $\Phi(\Ff_j(\Dbf))=\tau'$. Let $\ell=\max \tau_j$ be the label of the particle that jumps to site $j+1$ in $\tau'$. We have $\fa_1=j$, $\fa_2=j+1$, and $D_2=\Ff_j(D_1)$ (a particle moves from site $j$ to site $j+1$). There are two cases to consider: $j+1\not\in D_2$ or $j+1\in D_2$. We compare $\Phi(\Dbf)=D_1^{(1)}(\Phi(\mathbf{G})^{(+1)})$ to $\Phi(\mathbf{D'})=D_1'^{(1)}(\Phi(\mathbf{G'})^{(+1)})$ in both cases. Let us write $\bwt=\Phi(\mathbf{G})^{(+1)}$ and $\mathbf{\wt'}=\Phi(\mathbf{D'})^{(+1)}$ so that $\Phi(\Dbf)=D_1^{(1)}(\bwt)$ and $\Phi(\mathbf{D'})=D_1'^{(1)}(\mathbf{\wt'})$. Considering the pairing procedure, we notice that we can disregard the pairing of any particles of label greater than $\ell$ in $\bwt$ and $\mathbf{\wt'}$, since none of those particles pair to site $j$ in $D_1$ or $D_1'$, since $\ell$ is the largest paired particle in $D_1$ at site $j$. 
    
If $j+1\not\in D_2$, $\wt_{j+1}=\emptyset$ as well, so by the inductive hypothesis, $\Phi(\mathbf{G'})=\Phi(\Ff_{j+1}(\mathbf{G}))=\Phi(\mathbf{G})$ and so $\mathbf{\wt'}=\bwt$. Thus we compare $D_1^{(1)}(\bwt)$ to $D_1'^{(1)}(\bwt)$ and consider the rightmost particle with label $\ell$ in $\bwt$ that pairs to site $j$ in $D_1$. By assumption, this particle cannot be at site $j+1$ in $\bwt$, so it will be able to pair with the particle at site $j+1$ in $D_1'$. All other pairings between $\bwt$ and $D_1$ and $D_1'$ are identical, which means precisely that $\Phi(\mathbf{D'})=\tau'$.

If $j+1\in D_2$, $\wt_{j+1}\neq \emptyset$ as well. By the induction hypothesis, the particle with largest label jumps from site $j+1$ in $\Phi(\mathbf{G})$ to site $j+2$ to obtain $\Phi(\Ff_{j+1}(\mathbf{G}))$, and so $\mathbf{\wt'}$ is obtained from $\bwt$ by having the largest label particle at site $j+1$ jump to site $j+2$. If that particle has label smaller than $\ell$, the same situation as above will occur. If that particle has label $\ell$, it will necessarily pair with the particle that jumped from site $j$ to $j+1$ in row 1 of $D_1$. All other pairings from $\bwt$ to $D_1$ and from $\mathbf{\wt'}$ to $D_1'$ are identical, and so we again have $\Phi(\mathbf{D'})=\tau'$.
\end{proof}

\begin{prop}
    The bijections $\Ff,\Rf$ commute with involution $\sigma_j$ for all $j$, namely, for $\Dbf\in\bMLQ(\alpha,n)$,  $\sigma_j(\Ff_i(\Dbf,i))=\Ff_i(\sigma_j (\Dbf))$ and $\sigma_j(\Rf_i(\Dbf,i))=\Rf_i(\sigma_j (\Dbf))$.
\end{prop}

\begin{proof}
    Let $\Dbf=(D_1,\ldots,D_L)$, let $1\leq j<L$, and without loss of generality, assume $|D_j|>|D_{j+1}|$. Let $\sigma_i(\Dbf)=(D_1,\ldots,D_j',D_{j+1}',\ldots,D_L)$.  Suppose the ringing path of $\Dbf$ and $\sigma_j(\Dbf)$ triggered by site $i$ is $\fa_j=\fa'_j=u$ at row $j$ (the ringing paths are equal up to and including row $j$). There are two cases to consider: $u\not\in D_j$, in which case $\fa_{j+1}=u$, and $u\in D_j$, in which case $\fa_{j+1}=u+1$. Consider the sites $u, u+1$ in $D_j,D_{j+1}, D'_j$, and $D'_{j+1}$. 

    It should be noted that if all particles at sites $\fa_j$ and $\fa_{j+1}$ are paired, then $\fa'_{j+1}=\fa_{j+1}$ and $\fa'_{j+2}=\fa_{j+2}$, so the ringing paths are identical on $\Dbf$ and $\sigma_i(\Dbf)$, from which the claim immediately follows. Thus let us assume $\sigma_i$ acts nontrivially on either site $\fa_j$ in $D_{j+1}$ or site $\fa_{j+1}$ in $D_{j+1}$ (or both).

    In the first case, we have $\fa_{j+1}=u$, and we must assume $u\in D_{j+1}$, and that at least one of the particles at site $u$ of row $j+1$ is unpaired (or else $\sigma_i$ acts trivially on both $\fa_j$ and $\fa_{j+1}$ in $D_{j+1}$). Thus let us write $d_{j}=\iota(D_{j}),d_{j+1}=\iota(D_{j+1}),d'_{j}=\iota(D'_{j}),d'_{j+1}=\iota(D'_{j+1})$, and suppose $d_{j+1}(u)=k>0$ and $d'_j(u)=\ell\leq k$. Then $\fa'_{j+1}=u+1$ and $\fa_{j+2}=\fa_{j+1}+1=u+1$. Since pairing is to the left, if there are any unpaired particles at site $u$ of $D'_{j+1}$, then all particles (if any) at site $u+1$ are also unpaired. Then $u+1\not\in D'_{j+1}$, and so $\fa'_{j+2}=\fa'_{j+1}=u+1=\fa_{j+2}$, implying that the ringing paths on $\Dbf$ and $\sigma_i(\Dbf)$ are equal at rows $m\leq j$ and $m\geq j+2$. Moreover, $\Ff_i(\Dbf)_{j+1}(u)=d_{j+1}(u)-1$ and $\Ff_i(\Dbf)_{j+1}(u+1)=d_{j+1}(u+1)+1$, whereas  $\Ff_i(\Dbf')_{j}(u)=d'_{j}(u)-1$ and $\Ff_i(\Dbf')_{j}(u+1)=d'_{j}(u+1)+1$. Therefore, we can assume the particle that jumped from site $u$ of $D'_j$ is unpaired (since there is at least one unpaired particle at site $u$ of $D_{j+1}$), and so $\Dbf'=\sigma_j(\Ff_i(\sigma_j(\Dbf)))$.

    In the second case, we have $\fa_{j+1}=u+1$. If $u+1\not\in D_{j+1}$, then $u+1\not\in D'_{j+1}$, so the ringing paths and transitions are equal on $\Dbf$ and $\sigma_j(\Dbf)$. On the other hand, if $u+1\in D_{j+1}$, then there will necessarily be $\min\{d_j(u),d_{j+1}(u+1)\}$ paired particles at sites $u$ of $D_j$ and $u+1$ and $D_{j+1}$ in $\Dbf$, and thus those particles are also at sites $u$ of $D'_j$ and $u+1$ of $D'_{j+1}$ in $\sigma_j(\Dbf)$. Then $\fa'_{j+1}=\fa_{j+1}$ and $\fa'_{j+2}=\fa_{j+2}$, and so the ringing paths and transitions are again equal in $\Dbf$ and $\sigma_j(\Dbf)$. This completes the proof for $\Ff_i$, and a similar argument proves the claim for $\Rf_i$.
\end{proof}

\begin{remark}
    \label{rem:femion_MLQ_ringing}
    Unfortunately, the property above is not shared by the ringing paths on fermionic multiline queues in \cref{sec:straight MC}, and the Markov chain on twisted multiline queues $\MLQ(\alpha,n)$ for a composition $\alpha$ does not in general project to the TASEP of corresponding type. 

    For example, consider $\alpha=(2,1,1,2)$ $\Qbf = (\{1,2\}, \{2\},\emptyset, \{1\}) \in\MLQ(\alpha,4)$. The ringing path initiated at $j=1$ is $a_1=a_2=1,a_3=a_4=2$, so pictorally we have
    \[
    \Qbf = \ytableaushort{11{}{},{}1{}{},{}{}{}1,1{}11},
    \qquad\qquad
    F_1(\Qbf)=\ytableaushort{11{}{},1{}{}{},{}{}{}1,1{}11}.
    \]
    Then $\Phi(\Qbf) = 2031$ and $\Phi\bigl(F_1(\Qbf)\bigr) = 3021$, whereas there is clearly no transition from $2031$ to $3021$ in the TASEP of type $((3,2,1), 4)$. Thus~\eqref{eq:projection} does not hold for this pair of TASEP states.
    \end{remark}

We expect that in order for a Markov process on twisted multiline queues to project to the TASEP, it would need to commute with the involution $\sigma_i$. One such example would be to start with the ringing paths for straight multiline queues and then apply $\sigma_i$. Hence, we have the following question.

\begin{question}
What is an explicit description of the Markov chain on multiline queues that commutes with $\sigma_i$ and is given by the ringing paths for straight multiline queues?
\end{question}

\appendix
\section{Appendix}\label{appendix}

We include a description of the fermionic and bosonic projection maps that is equivalent to \cref{def:mlq pairing,def:mld pairing}, respectively, as a queueing algorithm in which a single particle is paired at a time. This formulation is analogous to the original Ferrari-Martin procedure, and is useful when considering the pairing process as a queueing system and keeping track of the order in which pairings occur.

In \cref{def:mlq pairing}, at each row, particles of the same label are paired simultaneously. It is immediate that pairing one particle at a time during the particle phase results in the same set of particles paired to in the row below. However, some care is needed to describe the collapsing phase when particles are processed one at a time.

\begin{defn}[Fermionic queueing algorithm]\label{def:mlq pairing particlewise} 
For $Q\subseteq [n]$ and an integer $\ell$, the action of $Q^{(\ell)} \colon \fermwords\rightarrow \fermwords$ can be described as follows. Let $\bw=w_1\cdots w_n\in\fermwords$ with $\ell<\min(\bw_+)$, and define $Q^{(\ell)}(\bw)=(y_1,\ldots,y_n)\in\fermwords$. 
The procedure is conducted in two phases: the pairing phase and the collapsing phase. To initialize the process, set $y_j=0$ for all $1\leq j\leq n$, and choose a priority-respecting order $i_1,\ldots,i_r$ on the $r=|w_+|$ non-zero elements of $w$ such that $w_{i_a}\geq w_{i_b}$ for all $a<b$.

\textbf{Pairing phase.} Let $k = \abs{Q}$ and let $j=\min\{k,r\}$. Let $u^{(0)}:=\{i_1,\ldots,i_j\}$. For $s=i_1,\ldots,i_j$ in that order, pair the particle $s$ \textbf{weakly to the right} to the unpaired particles in $Q$. If a particle at site $s$ pairs to $t\in Q$, set $y_{t}=w_{s}$. Once all $j$ particles in $u^{(0)}$ have been paired, if $k>r$, set $v_t=\ell$ for any remaining unpaired particles $t\in B$, completing the pairing phase. 

\textbf{Collapsing phase.} If $r>k$, $\{i_{k+1},\ldots,i_r\}$ are the indices of the unpaired nonzero elements in $w$ (still in priority order). For $j=1,2,\ldots,r-k$, set $u^{(j)}$ to be $u^{(j-1)}\cup \{i_{j+k}\}$, and consider $\Par(\cw(Q,u^{(j)}))$, which necessarily has exactly one index $m$ that is unmatched in $u^{(j)}$. Set $v_{m}=w_{i_{j+k}}-1$ and remove $m$ from $u^{(j)}$. 

\textbf{Output.} Once both phases are completed, define $Q^{(\ell)}(\bw)=(y_1,\ldots, y_n)$.
\end{defn}

\begin{example}
We give an example of the particle-wise pairing procedure given in \cref{def:mlq pairing particlewise}. Consider $n=6$, $Q=\{1,3,5,6\}$, $\ell=2$, and $\bw=(2,4,3,4,3,3)$. Choose the priority-respecting ordering of particles to be $(i_1,\ldots,i_6)=(2,4,3,5,6,1)$. In the pairing phase, particles $(i_1,\ldots,i_4)=(2,3,4,5)$ are paired to yield $(y_1,\ldots,y_6)=(3,0,4,0,4,3)$. For the collapsing phase, the unpaired indices are $(i_5,i_6)=(6,1)$.
\begin{center}
\begin{tabular}{c|c|l|c|c|l}
$j$&$i_{j+k}$&$u^{(j)}$&$m$&$w_{i_{j+k}}$&$Q^{(\ell)}(\bw)$\\\hline
$0$&&$u^{(0)}=\{2,3,4,5\}$&&&$(3,0,4,0,4,3)$\\
$1$&$i_5=6$&$u^{(1)}=\{2,3,4,5,6\}$&$2$&$w_{i_5}=3$&$(3,2,4,0,4,3)$\\
$2$&$i_6=1$&$u^{(2)}=\{1,3,4,5,6\}$&$4$&$w_{i_6}=2$&$(3,2,4,1,4,3)$\\
\end{tabular}
\end{center}
Thus $Q^{(2)}(\bw)=(3,2,4,1,4,3)$.
\end{example}

We give an analogous description of the pairing algorithm for bosonic multiline queues.

\begin{defn}[Bosonic queueing algorithm]\label{def:mld pairing particlewise}
For $D\in \mathcal{M}(n)$ and an integer $\ell$, the action of $D^{(\ell)}\colon \bosonwords\rightarrow \bosonwords$ can be described as follows.
Let $\bwt=(\wt_1,\ldots,\wt_n)\in\bosonwords$ with $\ell<\min\cup_{i=1}^n \wt_i$. Then $D^{(\ell)}(\bwt)=(Y_1,\ldots,Y_n)\in\bosonwords$ where $Y_j$ is initialized to be $\emptyset$ for each $j$. The procedure is conducted in two phases: the particle phase and the collapsing phase. Choose an arbitrary ordering $\{t_1,\ldots,t_r\}$ on the $r=|\bwt|$ elements of $\bwt$, and define $\bwt^{-1} \colon [r]\rightarrow [n]$ by $\bwt^{-1}(j)=u$ if $t_j\in \wt_u$ (i.e.~the site in $\bwt$ containing the $j$'th particle according to this ordering).  Choose a priority-respecting order $i_1,\ldots,i_r$ on these elements such that $t_{i_a}\geq t_{i_b}$ for all $a<b$.

\textbf{Particle phase.} Let $k=|D|$ and let $j=\min\{k,r\}$. Let $\widetilde{v}^{(0)}:=\{\bwt^{-1}(i_1),\ldots,\bwt^{-1}(i_j)\}$ be the multiset of indices in $\{1,\ldots,n\}$ that contain the particles indexed by $\{i_1,\ldots,i_j\}$. For $s=i_1,\ldots,i_j$ in that order, pair the particle $s$ (corresponding to $\bwt^{-1}(s)$ in $\widetilde{v}^{(0)}$) \textbf{strictly to the left} to the unpaired particles in $D$. If a particle $s$ pairs to $b\in D$, append $t_s$ to $Y_b$. Once all $j$ particles in $\widetilde{v}^{(0)}$ have been paired, if $k>r$, append $\ell$ to $Y_i$ for each remaining unpaired particle $i\in D$, completing the particle phase. 

\textbf{Collapsing phase.} If $r>k$, $\{i_{k+1},\ldots,i_r\}$ are the indices of the unpaired elements in $\bwt$ (still in priority order). For $j=1,2,\ldots,r-k$, set $\widetilde{v}^{(j)}$ to be $\widetilde{v}^{(j-1)}\cup \{\bwt^{-1}(i_j)\}$, and consider $\bPar(D,\widetilde{v}^{(j)})$, which necessarily has exactly one index $m$ that is unmatched in $\widetilde{v}^{(j)}$. Append $t_{i_j}-1$ to $Y_{m}$ and remove $\bwt^{-1}(m)$ from $\widetilde{v}^{(j)}$.

\textbf{Output.} Once both phases are completed, define $D^{(\ell)}(\bwt)=(Y_1,\ldots,Y_n)$.
\end{defn}

\begin{example}
Consider $n=5$, $\Dbf=\{1,1,1,3,3,5,5\}$, $\ell=2$, and $\bwt=(\{t_1=2,t_2=3,t_3=3,t_4=4\},\{t_5=3,t_6=4,t_7=4\},\emptyset,\{t_8=2\},\{t_9=3,t_{10}=4,t_{11}=4\})$. Choose the priority-respecting ordering $(i_1,\ldots,i_{11})=(4,6,7,10,11,2,3,5,9,1,8)$. During the pairing phase, particles $(i_1,\ldots,i_7)=(4,6,7,10,11,2,3)$ corresponding to $\widetilde{u}^{(0)}=\{1,1,1,2,2,5,5\}$ are paired to yield $(Y_1,\ldots,Y_5)=(44,\emptyset,344,\emptyset,34)$. For the collapsing phase, the unpaired indices are $(i_8,\ldots,i_{11})=(5,9,1,8)$ which $\bwt^{-1}$ maps to $(2,5,1,4)$.
\begin{center}
\begin{tabular}{c|c|l|c|c|l}
$j$&$\bwt^{-1}(i_{j+k})$&$\widetilde{v}^{(j)}$&$m$&$t_{i_{j+k}}$&$(Y_1,Y_2,Y_3,Y_4,Y_5)$\\\hline
$0$&&$\widetilde{v}^{(0)}=\{1,1,1,2,2,5,5\}$&&&$(44,\emptyset,344,\emptyset,34)$\\
$1$&$\bwt^{-1}(i_8)=2$&$\widetilde{v}^{(1)}=\{1,1,1,2,2,2,5,5\}$&$1$&$t_{i_8}=3$&$(344,\emptyset,344,\emptyset,34)$\\
$2$&$\bwt^{-1}(i_9)=5$&$\widetilde{v}^{(2)}=\{1,1,2,2,2,5,5,5\}$&$5$&$t_{i_9}=3$&$(344,\emptyset,344,\emptyset,334)$\\
$3$&$\bwt^{-1}(i_{10})=1$&$\widetilde{v}^{(3)}=\{1,1,1,2,2,2,5,5\}$&$1$&$t_{i_{10}}=2$&$(2344,\emptyset,344,\emptyset,334)$\\
$4$&$\bwt^{-1}(i_{11})=4$&$\widetilde{v}^{(4)}=\{1,1,2,2,2,4,5,5\}$&$5$&$t_{i_{11}}=2$&$(2344,\emptyset,344,\emptyset,2334)$\\
\end{tabular}
\end{center}
Thus $\D^{(\ell)}(\bwt)=(2344,\emptyset,344,\emptyset,2334)$.
\end{example}

\bibliographystyle{plain}
\bibliography{biblio}

\end{document}